\documentclass[a4paper,10pt,leqno]{amsart}
\usepackage{amssymb, amsmath}
\newtheorem{thm}{Theorem}[section]
\newtheorem{lem}[thm]{Lemma}

\newtheorem{prop}[thm]{Proposition}
\newtheorem{defn}[thm]{Definition}
\theoremstyle{definition}
\newtheorem{rmk}{Remark}
\newtheorem{exa}{Example}
\numberwithin{equation}{section}
\usepackage[foot]{amsaddr}
\allowdisplaybreaks
\usepackage{color}

\newcommand{\<}{\langle}
\renewcommand{\>}{\rangle}
\DeclareMathOperator{\supp}{supp}
\DeclareMathOperator{\dd}{div}
\DeclareMathOperator{\Ker}{Ker}
\newcommand{\pp}{\partial}
\newcommand{\M}{\mathcal{M}}
\newcommand{\W}{\mathcal{W}}
\newcommand{\D}{\mathcal{D}}
\newcommand{\vae}{\varepsilon}
\newcommand{\vap}{\varphi}
\newcommand{\la}{\lambda}
\newcommand{\ue}{u^{\varepsilon}}

\newcommand{\un}{u^{0}}
\newcommand{\wto}{\rightharpoonup}

\newcommand{\per}{\mathrm{per}}
\newcommand{\utext}[1]{\text{\upshape #1}}
\usepackage{blkarray}
\title[]{%
Homogenization and
Inverse Problems 
for Fractional Diffusion Equations 
}
\author[A.~Kawamoto]{Atsushi Kawamoto%
${}^1$%
}
\address{%
${}^1$%
Faculty of Engineering, Takushoku University, 
Tatemachi, Hachioji, Tokyo 193-0985, Japan%
}
\email{%
${}^1$%
akawamot@la.takushoku-u.ac.jp%
}
\author[M.~Machida]{Manabu Machida%
${}^2$%
}
\address{%
${}^2$%
Institute for Medical Photonics Research, 
Hamamatsu University School of Medicine,
Hamamatsu, Shizuoka 431-3192, Japan%
}%
\email{%
${}^2$%
machida@hama-med.ac.jp%
}
\author[M.~Yamamoto]{Masahiro Yamamoto%
${}^{3,4,5}$%
}
\address{%
${}^3$%
Graduate School of Mathematical Sciences, 
The University of Tokyo, 
Komaba, Meguro, Tokyo 153-8914, Japan%
}%
\address{%
${}^4$%
Honorary Member of Academy of Romanian Scientists, 
Ilfov, nr. 3, Bucuresti, Romania \\
}%
\address{%
${}^5$%
Correspondence member of Accademia Peloritana dei Pericolanti
Palazzo Universit\`a, Piazza S. Pugliatti 1 98122 Messina Italy
}%
\email{%
${}^{3,4,5}$%
myama@ms.u-tokyo.ac.jp%
}

\date{}

\begin{document}
%
%
\begin{abstract}
We consider the homogenization for time-fractional diffusion equations in a periodic structure and 
we derive the homogenized time-fractional diffusion equation. 
Then we discuss the determination of the constant diffusion coefficient by minimum data. 
Combining the results obtained above, we investigate the inverse problems of 
determining the diffusion coefficient 
for the homogenized equations from the data in the periodic structure and vice versa, that is, 
we investigate 
the inverse problem of 
determining the diffusion coefficient 
for the periodic equations from the data in the homogenized structure.
\end{abstract}

\maketitle


%
%

\section{Introduction}

Let $0<\alpha<1$, $N\in\{1,2,3\}$ and $T>0$.
Let $\Omega\subset\mathbb{R}^N$ be a bounded domain with the boundary $\pp\Omega$ of $C^2$-class. 
We consider the following time-fractional diffusion equation 
in a periodic structure
with the parameter $\vae>0$, 
the zero Dirichlet boundary condition
and 
the initial condition. 
\begin{equation}
\label{eq:101}
\left\{
\begin{aligned}
&
\pp_t^\alpha \ue(x,t)
-
\dd (A^\vae(x)\nabla \ue(x,t))
=
f^\vae(x,t),
&(x,t)\in \Omega\times(0,T)
,\\&
\ue(x,t)=0,
&(x,t)\in\pp\Omega\times(0,T)
,\\&
\ue(x,0)=\ue_0(x),
&x\in\Omega
,
\end{aligned}
\right.
\end{equation}
where $A^\vae(x)=(a_{ij}^\vae(x))_{i,j=1,\ldots,N}$ is 
the symmetric matrix valued function 
which satisfies appropriate conditions. 
The functions $u_0^\vae$ and $f^\vae$ will be defined later.  
Here $\pp_t^\alpha$ is 
the fractional derivative of Caputo type. 

First, 
we consider the homogenization 
for the initial boundary value problem for the time-fractional diffusion equation in \S3.1. 
To discuss the homogenization,
we use the variational formulation of the above problem \eqref{eq:101} in $L^2$-sense 
introduced by 
Kubica, Ryszewska and Yamamoto \cite{KRY2020}. 
We consider the problem \eqref{eq:101} as a problem in the periodic structure. 
Then we may obtain the following homogenized 
problem.  
\begin{equation}
\label{eq:102}
\left\{
\begin{aligned}
&
\pp_t^\alpha \un(x,t)
-
\dd (A^0\nabla \un(x,t))
=
f(x,t),
&(x,t)\in\Omega\times(0,T)
,\\&
\un(x,t)=0,
&(x,t)\in\pp\Omega\times(0,T)
,\\&
\un(x,0)=u_0(x),
&x\in\Omega
,
\end{aligned}
\right.
\end{equation}
where $A^0$ is the homogenized matrix. 
The functions $u_0$ and $f$ are the limits of $\{u_0^\vae\}$ and $\{f^\vae\}$ 
in
appropriate senses. 
As a previous research 
on the homogenization for time-fractional diffusion equation, 
we may refer to the article \cite{JL2020}, 
in which the homogenized equations \eqref{eq:102} are derived by 
the multiple scale method via the formal calculation. 
In this article, we discuss the homogenization by a rigorous argument. 
We establish the homogenization result 
by the oscillating test function method. 
As for the monograph on the homogenization, we may refer to Cioranescu and Donato \cite{CD1999} 
and references therein. 

Second, we investigate the inverse problem 
of determining the coefficient 
for the homogenized equations in \S3.2.
We begin with the case where the diffusion coefficient is a scaler. 
Then for $N\in\{2,3\}$, we consider layered materials
in which 
the diffusion coefficient depends on 
a single spatial
variable. 

%
%
Third, we propose new inverse problems 
between different structures
by combining the method to solve inverse problems with the homogenization in \S3.3. 
The method to solve inverse problems gives 
the uniqueness and the stability in 
the determination of coefficient from observation data in the same structures. 
Meanwhile, the homogenization yields 
homogenized equations from periodic equations. 
Using the knowledge of the above two issues, 
we may consider the inverse problems between the periodic structure and homogenized structure. 
That is, we investigate 
the inverse problems of 
determining the diffusion coefficient 
for the homogenized equations from the data in the periodic structure and  
the inverse problem of 
determining the diffusion coefficient 
for the periodic equations from the data in the homogenized structure.

%

%
%

In section 2, 
we prepare the known results and tools on the 
well-posedness for time-fractional diffusion equations. We also prepare 
the basic tools of 
the homogenization with periodic structures. 
In section 3, 
we state our main results. 
We start with 
the homogenization result for time-fractional diffusion equations. 
We discuss the determination of the constant diffusion coefficient. 
Then, we state the uniqueness and stability result in our inverse problem between different structures. 
In section 4, 
We prove our main results. 

%
%
\section{Preliminaries}

In this section, 
to make the content of this article self-contained, 
we prepare some tools and state previous results 
which we use to prove our main results. 
However, we will keep our preparation to a minimum 
to shorten this article and leave much of the contents to the monographs and papers.
We will cite the monographs and papers in each subsection as appropriate.

%
%
\subsection{On the Well-posedness for Fractional Diffusion Equations}

We state the previous results 
on the well-posedness 
for time-fractional diffusion equations
by the $L^2$-framework. 
We use these results to prove the our main theorem on the homogenization. 
For more details, 
We refer to the monograph \cite{KRY2020} by 
Kubica, Ryszewska and Yamamoto.

%
%
\subsubsection{Function Spaces and the General Fractional Derivative of Caputo Type}

In the classical sense, 
the Caputo type fractional derivative is defined by
\begin{equation*}
d_t^\alpha v (t)
=
\frac{1}{\Gamma(1-\alpha)}
\int_0^t
(t-\tau)^{-\alpha}
\frac{dv}{d\tau}(\tau)
\,d\tau,
\quad 0\leq t\leq T
\end{equation*}
for $v\in W^{1,1}(0,T)$, 
where $\Gamma$ denotes the gamma function. 
In order to state the well-posedness result 
for the time-fractional diffusion equations by $L^2$-framework, 
we require that the above fractional derivative should be included in $L^2$ space. 
However, 
the $W^{1,1}$ space is small for us, 
since $d_t^\alpha v\in L^1(0,T)$ if $v\in W^{1,1}(0,T)$ in the above definition of $d_t^\alpha$. 
For this purpose, 
we redefine 
the general fractional derivative of Caputo type 
in an appropriate function space.

First, we define the function space $H_\alpha(0,T)$ 
which will be the domain of general fractional differential operator $\pp_t^\alpha$ of Caputo type. 
We know the Sobolev-Slobodecki space $H^\alpha(0,T)$ with the norm (see e.g., Adams \cite{A1975}):
\begin{equation*}
\|u\|_{H^\alpha(0,T)}
=
\left(
\|u\|_{L^2(0,T)}^2
+
\int_0^T\int_0^T
\frac{|u(t)-u(s)|^2}{|t-s|^{1+2\alpha}}
\,ds\,dt
\right)^{\frac12}
.
\end{equation*}
Using this function space $H^\alpha(0,T)$, 
we define the Banach space $H_\alpha(0,T)$ by 
\begin{equation*}
H_{\alpha}(0,T)
=
\left\{
\begin{aligned}
&
H^\alpha(0,T),
&
0\leq \alpha<\frac12,\\
&
\left\{
u\in H^{\frac12}(0,T)\,\middle|\,
\int_0^T
\frac{|u(t)|^2}{t}\,dt<\infty
\right\},
&
\alpha=\frac12,\\
&
\{u\in H^\alpha(0,T)
\mid
u(0)=0
\},
&
\frac12<\alpha\leq 1
\end{aligned}
\right.
\end{equation*}
with the norm
\begin{equation*}
\|u\|_{H_\alpha(0,T)}
=
\left\{
\begin{aligned}
&
\|u\|_{H^\alpha(0,T)},
&0\leq\alpha\leq 1,\,\alpha\neq\frac12,\\
&
\left(
\|u\|_{H^{\frac12}(0,T)}^2
+
\int_0^T
\frac{|u(t)|^2}{t}\,dt
\right)^{\frac12},
&
\alpha=\frac12
.
\end{aligned}
\right.
\end{equation*}

Let us define the general fractional derivative of Caputo type. 
To do this, 
we introduce the Riemann-Liouville fractional integral operator
\begin{equation*}
J^\beta v(t)
=
\frac{1}{\Gamma(\beta)}\int_0^t (t-s)^{\beta-1} v(s)\,ds
,\quad
v\in L^2(0,T)
\end{equation*}
for $\beta>0$. 
Then $J^\alpha:L^2(0,T)\to H_\alpha(0,T)$ is homeomorphic 
(see Theorem 2.1 in \cite{KRY2020} 
and
Theorem 2.1 in \cite{GLY2015}%
) 
and 
this operator $J^\alpha$ 
coincides with the $\alpha$-th fractional power of the following operator $J$ defined by 
\begin{equation*}
J v(t)=
\int_0^t v(s)\,ds
,\quad
v\in L^2(0,T). 
\end{equation*}
See Lemma 2.4 in \cite{KRY2020}.
Now we ready to state the definition of general fractional derivative of Caputo type. 
\begin{defn}[General Fractional Derivative of Caputo Type]
\label{defn:2101}
We define $\pp_t^\alpha$ by 
$\pp_t^\alpha:=(J^\alpha)^{-1}$ with the domain 
$\D(\pp_t^\alpha)=H_\alpha(0,T)$.
\end{defn}
Then we see that $\pp_t^\alpha:H_\alpha(0,T)\to L^2(0,T)$ is homeomorphic 
and the following norm equivalence holds:
\begin{equation}
\label{eq:2101}
C^{-1}
\|v\|_{H_\alpha(0,T)}
\leq
\|\pp_t^\alpha v\|_{L^2(0,T)}
\leq
C
\|v\|_{H_\alpha(0,T)}
,\quad v\in H_\alpha(0,T)
\end{equation}
with the constant $C>0$ depending on only $\alpha$. 
Moreover we have
\begin{equation}
\label{eq:2102}
\pp_t^\alpha v
=
\frac{d}{dt}(J^{1-\alpha} v)
,\quad
v\in H_\alpha(0,T)
.
\end{equation}
As for these results on the general fractional derivative of Caputo type, 
we may refer to Theorem 2.4 in \cite{KRY2020}.

%
%
\subsubsection{Well-posedness Result}

Let $\nu,\mu\in\mathbb{R}$ satisfy $0<\nu<\mu$.  
We denote by $M_S(\nu,\mu,\Omega)$ 
the set of the $N\times N$ matrix 
$A(x)=(a_{ij}(x))_{i,j=1,\ldots,N}$ satisfying 
the following conditions:
for any $\xi\in\mathbb{R}^N$,
\begin{equation*}
\left\{
\begin{aligned}
&\text{i)}
&&a_{ij}\in L^\infty(\Omega)
,\\
&\text{ii)}
&&\text{$A$ is symmetric, that is, $a_{ij}=a_{ji}$}
,\\
&\text{iii)}
&&(A(x)\xi,\xi)\geq \nu|\xi|^2,\quad x\in\overline{\Omega}
,\\
&\text{iv)}
&&|A(x)\xi|\leq \mu |\xi|,\quad x\in\overline{\Omega}
.
\end{aligned}
\right.
\end{equation*}

Let 
$A\in M_S(\nu,\mu,\Omega)$, 
$f\in L^2(0,T;H^{-1}(\Omega))$, 
$u_0\in L^2(\Omega)$. 
We consider the following initial boundary value problem. 
\begin{equation}
\label{eq:S2_P}
\left\{
\begin{aligned}
&
\pp_t^\alpha u(x,t)
-
\dd (A(x)\nabla u(x,t))
=
f(x,t),
&(x,t)\in\Omega\times(0,T)
,\\&
u(x,t)=0,
&(x,t)\in\pp\Omega\times(0,T)
,\\&
u(x,0)=u_0(x),
&x\in\Omega
.
\end{aligned}
\right.
\end{equation}
We define the function space
\begin{equation*}
\W^\alpha(u_0)
=
\{
v\mid
v\in L^2(0,T;H_0^1(\Omega)), 
v-u_0\in H_\alpha(0,T;H^{-1}(\Omega))
\}
\end{equation*}
with the norm
\begin{equation*}
\|v\|_{\W^\alpha(u_0)}
=
\|v\|_{L^2(0,T;H_0^1(\Omega))}
+
\|v-u_0\|_{H_\alpha(0,T;H^{-1}(\Omega))}
.
\end{equation*}
We investigate 
the weak solution of the problem \eqref{eq:S2_P} 
which is obtained by the following 
variational formulation
\begin{equation}
\label{eq:S2_wP}
\left\{
\begin{aligned}
&
\text{Find $u\in \W^\alpha(u_0)$ such that}\\
&
\left\<
\pp_t^\alpha (u-u_0)
,
\Phi
\right\>_{H^{-1}(\Omega),H_0^1(\Omega)}
+
\left(
A\nabla u,
\nabla \Phi
\right)_{L^2(\Omega)}
\\&
=
\left\<
f
,
\Phi
\right\>_{H^{-1}(\Omega),H_0^1(\Omega)}
\quad
\text{for $\Phi\in H_0^1(\Omega)$}
,\quad
\text{a.e. $t\in (0,T)$}
.
\end{aligned}
\right.
\end{equation}
In this formulation, 
we note that $u-u_0\in H_\alpha(0,T;H^{-1}(\Omega))$ 
contains the initial condition for $\frac12<\alpha<1$. 
Indeed, we see that 
$u-u_0\in C([0,T];H^{-1}(\Omega))$
by $H^\alpha(0,T)\subset C[0,T]$. 
Then we have 
$u(x,0)=u_0(x)$ 
in the sense of 
\begin{equation*}
\lim_{t\to 0} u(\cdot, t)=u_0
\ \text{in $H^{-1}(\Omega)$}
.
\end{equation*}

\begin{thm}
\label{thm:well-posedness_FDE}
Assume that 
$A\in M_S(\nu,\mu,\Omega)$, 
$f\in L^2(0,T;H^{-1}(\Omega))$, 
$u_0\in L^2(\Omega)$. 
Then there exists a unique weak solution of the problem \eqref{eq:S2_P}. 
Furthermore, we have the following estimate. 
\begin{equation}
\label{eq:estimate01}
\|u\|_{\W^\alpha(u_0)}
\leq
C
\left(
\|u_0\|_{L^2(\Omega)}
+
\|f\|_{L^2(0,T;H^{-1}(\Omega))}
\right)
.
\end{equation}
\end{thm}
On the proof of this theorem, we may refer to 
Kubica and Yamamoto \cite{KY2018}. 
If we assume that $A$ is smooth enough, 
we may prove this theorem more easily
(see the monograph \cite{KRY2020}). 

Related to the function space $\W^\alpha(u_0)$, 
we will prove the following lemma. 
This is a direct consequence of the results by Simon \cite{S1986}. 
\begin{lem}\label{lem:2101}
We consider the Banach space
\begin{equation*}
\mathrm{W}^\alpha
=
\{
v\mid
v\in L^2(0,T;H_0^1(\Omega)), 
v\in H^\alpha(0,T;H^{-1}(\Omega))
\}
\end{equation*}
with the norm 
\begin{equation*}
\|v\|_{\mathrm{W}^\alpha}
=
\|v\|_{L^2(0,T;H_0^1(\Omega))}
+
\|v\|_{H^\alpha(0,T;H^{-1}(\Omega))}
.
\end{equation*}
Then the embedding
\begin{equation*}
\mathrm{W}^\alpha\hookrightarrow L^2(0,T;L^2(\Omega))
\end{equation*}
is the compact embedding. 
\end{lem}
\begin{proof}
Let $\{v^\vae\}\subset \mathrm{W}^\alpha$ 
be a bounded sequence for $\vae>0$. 
That is, 
there exists $M>0$ such that 
$\|v^\vae\|_{\mathrm{W}^\alpha}\leq M$ for any $\vae>0$. 

Applying the Lemma 5 in Simon \cite{S1986} 
(with $\mu=\alpha, r=2, p=2$), we see that
$v^\vae\in L^2(0,T;H^{-1}(\Omega))$ satisfies 
\begin{equation*}
\|\tau_h v^\vae-v^\vae\|_{L^2(0,T-h;H^{-1}(\Omega))}
\leq 
CM h^\alpha
\end{equation*}
for $h>0$, where 
$C>0$ is a constant which is independent of $v^\vae$, 
and $\tau_h$ is the translation defined by 
$(\tau_h v)(x,t)=v(x,t+h)$ for $h>0$. 

Note that $H_0^1(\Omega)\subset L^2(\Omega)\subset H^{-1}(\Omega)$ and 
$H_0^1(\Omega)\hookrightarrow L^2(\Omega)$ 
is the compact embedding. 
By this relation of function spaces and the above inequality, 
the assumptions of 
the Theorem 5 in Simon \cite{S1986} are satisfied 
(with $p=2$, $X=H_0^1(\Omega), B=L^2(\Omega), Y=H^{-1}(\Omega)$ and $F=\{v^\vae\}$). 
Hence, it follows that 
$\{v^\vae\}$ is relatively compact in $L^2(0,T;L^2(\Omega))$. 
Thus, we see that 
$\mathrm{W}^\alpha\hookrightarrow L^2(0,T;L^2(\Omega))$
is the compact embedding and we completed the proof. 
\end{proof}
\begin{rmk}
We will use this lemma to prove our homogenization result. 
This kind of result is called 
the Aubin-Lions-Simon Lemma, and 
widely used to study partial differential equations. 
For instance, the Aubin-Lions-Simon Lemma is used 
to prove the existence of solutions of nonlinear initial boundary value problems. 
\end{rmk}
\begin{rmk}
\label{rmk:2102}
Let $u_0\in L^2(\Omega)$. 
Then we have
$\W^\alpha(u_0)\subset \mathrm{W}^\alpha$ and 
the following inequality: 
there exists a constant $C>0$ such that 
\begin{equation*}
\|v\|_{\mathrm{W}^\alpha}
\leq
\|v\|_{\W^\alpha(u_0)}+
C
\|u_0\|_{L^2(\Omega)}
\end{equation*}
for $v\in \W^\alpha(u_0)$. 
\end{rmk}

%
%
\subsection{On the Homogenization}

We will obtain the homogenization results 
via the rigorous argument by using the functional analysis. 
Here we prepare the basic tools on the homogenization 
which will be used 
in the proof of our results. 
These are based on the monograph \cite{CD1999} by Cioranescu and Donato.

%
%
\subsubsection{Rapidly Oscillating Periodic Function}
First we see the definition and the weak limit of 
rapidly oscillating periodic functions. 

Let $Y$ be a reference cell 
defined by
\begin{equation*}
Y=
(0,\ell_1)\times\cdots\times(0,\ell_N)
,
\end{equation*}
where $\ell_i>0$ ($i=1,\ldots,N$). 
Let $\{e_1,\ldots,e_N\}$ be the canonical basis on $\mathbb{R}^N$  

\begin{defn}[$Y$-periodic]
A function $g$ defined a.e.\ on $\mathbb{R}^N$ is called $Y$-periodic  
if and only if 
\begin{equation*}
g(y)=g(y+k\ell_i e_i),\quad \utext{a.e.}\ y\in \mathbb{R}, 
\quad k\in \mathbb{Z},
\quad i\in\{1,\ldots,N\}.
\end{equation*}
\end{defn}

We define the mean value of $g\in L^1(Y)$ 
by
\begin{equation*}
\M_Y(g)=
\frac1{|Y|}
\int_Y g(y)\,dy.
\end{equation*}
Then 
we may obtain the following lemma.
\begin{lem}
\label{lem:2201}
Let $g$ be a $Y$-periodic function in $L^2(Y)$. 
For $\vae>0$, we set
\begin{equation*}
g^\vae(x)=g\left(\frac{x}{\vae}\right)
\ \utext{a.e. on $\mathbb{R}^N$}.
\end{equation*}
Then we have 
the following weak convergence:
for any bounded domain $\omega\subset\mathbb{R}^N$, 
\begin{equation*}
g^\vae
\wto
\M_Y(g)
\ \utext{weakly in $L^2(\omega)$},  
\end{equation*}
as $\vae\to 0$. 
\end{lem}
This $\vae Y$-periodic function $g^\vae$ is 
called the rapidly oscillating periodic function because it oscillates rapidly when $\vae$ tends to $0$.
We may refer to Theorem 2.6 in \cite{CD1999} 
for the proof of the above lemma.

Throughout this article,
we assume that 
the matrix $A$ satisfies the followings  
in the context of the homogenization: 
\begin{equation}
\label{eq:2201}
\left\{
\begin{aligned}
&
\text{$a_{ij}$ is $Y$-periodic for all $i,j=1,\ldots,N$}, 
\\
&
A=(a_{ij})\in M_S(\nu,\mu, Y)
,
\end{aligned}
\right.
\end{equation}
where $Y$ is the reference cell defined above. 
We set 
\begin{equation}
\label{eq:2202}
\left\{
\begin{aligned}
&
a_{ij}^\vae(x)
=
a_{ij}\left(\frac{x}{\vae}\right),\quad 
x\in\mathbb{R}^N,
\\
&
A^\vae(x)=A\left(\frac{x}{\vae}\right)
=
(a_{ij}^\vae)_{i,j=1,\ldots,N},\quad 
x\in\mathbb{R}^N.
\end{aligned}
\right.
\end{equation}
We will consider the periodic equations with 
this matrix $A^\vae$ in the homogenization. 
%
%
\subsubsection{Auxiliary Problems}

To prove our main result on the homogenization 
for time-fractional diffusion equations, 
we adopt the oscillating test function method by Tartar \cite{T1978} in which 
we consider the solution of the auxiliary periodic boundary value problem.
We will start with the following problem on the reference cell $Y$:
for any $\xi\in\mathbb{R}^N$, 
\begin{equation}
\label{eq:2203}
\left\{
\begin{aligned}
&
-\dd_y (A(y)\nabla_y \chi_\xi)=-\dd_y(A(y)\xi)\quad \text{in $Y$},
\\&
\text{%
$\chi_\xi$ is $Y$-periodic and 
$\M_Y(\chi_\xi)=0$%
},
\end{aligned}
\right.
\end{equation}
where $A$ satisfies the condition \eqref{eq:2201}. 
The variational formulation of this problem \eqref{eq:2203} is written by
\begin{equation}
\label{eq:2204}
\left\{
\begin{aligned}
&
\text{Find $\chi_\xi\in W_\per(Y)$ such that}
\\&
\int_Y A\nabla_y \chi_\xi\cdot\nabla_y v\,dy
=
\int_Y A\xi\cdot\nabla_y v\,dy
,\quad
\text{for all $v\in W_\per(Y)$},
\end{aligned}
\right.
\end{equation}
where the function space $W_\per(Y)$ is defined by
\begin{equation*}
W_\per (Y)
=
\{
v\in H_\per^1(Y)
\mid
\M_Y(v)=0
\}.
\end{equation*}
Here $H^1_\per(Y)$ is the following function space:
\begin{equation*}
H^1_\per(Y)
=
\overline{C_\per^\infty(Y)}^{H^1(Y)}, 
\end{equation*}
where 
$C_\per^\infty(Y)$ 
is the space of 
functions in 
$C^\infty(\mathbb{R^N})$ that are $Y$-periodic. 
For more details on the function space $H_\per^1(Y)$, we may refer to \S3.4 in the monograph \cite{CD1999}. 

It is known that 
there exists a unique solution $\chi_\xi\in W_\per(Y)$ of the problem \eqref{eq:2204} by 
Theorem 4.27 in \cite{CD1999}. 

Now we extend 
$\chi_\xi$ to $\mathbb{R}^N$ by periodicity 
(still denoted by $\chi_\xi$ the extension ). 
Then we see that this $\chi_\xi$ is the 
unique solution of the following problem (see Theorem 4.28 in \cite{CD1999}):
for any $\xi\in\mathbb{R}^N$, 
\begin{equation}
\label{eq:2205}
\left\{
\begin{aligned}
&
-\dd_y (A(y)\nabla_y \chi_\xi)=-\dd_y(A(y)\xi)\quad \text{in $\D^\prime(\mathbb{R}^N)$},
\\&
\text{%
$\chi_\xi$ is $Y$-periodic and 
$\M_Y(\chi_\xi)=0$%
}.
\end{aligned}
\right.
\end{equation}

Next we consider a function $w_\xi$
defined by $\chi_\xi$:
\begin{equation}
\nonumber
w_\xi=\xi\cdot y -\chi_\xi
\end{equation}
for any $\xi\in\mathbb{R}^N$.
Then,
we may see that 
$w_\xi$ is the unique solution of the following problem:
for any $\xi\in\mathbb{R}^N$, 
\begin{equation}
\label{eq:2208}
\left\{
\begin{aligned}
&
-\dd_y (A(y)\nabla_y w_\xi)=0
\quad \text{in $\D^\prime(\mathbb{R}^N)$},
\\&
\text{%
$w_\xi-\xi\cdot y$ is $Y$-periodic and 
$\M_Y(w_\xi-\xi\cdot y)=0$%
}.
\end{aligned}
\right.
\end{equation}

We will use the following function $w_\xi^\vae$  as an oscillating test function in the proof of our result on the homogenization. 
\begin{equation*}
w_\xi^\vae(x)
=
\vae w_\xi\left(\frac{x}{\vae}\right)
=
\xi\cdot x-\vae\chi_\xi\left(\frac{x}{\vae}\right)
,\quad x\in \Omega
\end{equation*}
for $\xi\in\mathbb{R}^N$, 
where $w_\xi$ and $\chi_\xi$ is the functions defined as above. 
We observe the convergences of the function $w_\xi^\vae$. 
Since $\chi_\xi$ is $Y$-periodic, 
we have
\begin{equation*}
w_\xi^\vae
\wto
\xi\cdot x
\ \text{weakly in $L^2(\Omega)$}
\end{equation*}
by Lemma \ref{lem:2201}. 
Noting that $w_\xi=-\chi_\xi+\xi\cdot y$, 
we obtain
\begin{equation*}
(\nabla_x w_\xi^\vae)(x)
=
(\nabla_y w_\xi)\left(\frac{x}{\vae}\right)
=
\xi-\nabla_y\chi_\xi\left(\frac{x}{\vae}\right). 
\end{equation*}
Since $\chi_\xi$ is $Y$-periodic, 
we see that $\nabla_y w_\xi$ is $Y$-periodic, too. 
Using Lemma \ref{lem:2201} again, 
we may get 
\begin{equation*}
\nabla_x w_\xi^\vae
\wto
\M_Y(\xi-\nabla_y\chi_\xi)
=
\xi-\M_Y(\nabla_y\chi_\xi)
\ \text{weakly in $L^2(\Omega)$}
.
\end{equation*}
By Green's formula, we have
\begin{equation*}
\int_Y
\nabla_y \chi_\xi(y)
\,dy
=
\int_{\pp Y}
\chi_\xi\cdot n
\,dS_y
=0,
\end{equation*}
where $n$ denotes the unit outward normal vector to $\pp Y$. 
Here we use the fact 
that
the traces of $\chi_\xi$ are same on the opposite faces of $Y$ 
(see Proposition 3.49 in \cite{CD1999}). 
Hence we obtain 
\begin{equation*}
\M_Y(\nabla_y\chi_\xi)=0. 
\end{equation*}
Thus 
we see that 
there exists the subsequence of $\{w_\xi^\vae\}$ 
(still denoted by $\vae$) satisfying  
the following convergences
\begin{equation}
\label{eq:2209}
\left\{
\begin{aligned}
&
w_\xi^\vae
\wto
\xi\cdot x
\ \text{weakly in $H^1(\Omega)$},
\\&
w_\xi^\vae
\to
\xi\cdot x
\ \text{strongly in $L^2(\Omega)$}.
\end{aligned}
\right.
\end{equation}
Here we used the Sobolev embedding theorem 
(see e.g., Adams \cite{A1975}), 
that is , the fact that 
the embedding $H^1(\Omega)\hookrightarrow L^2(\Omega)$ is the compact embedding. 
The above convergences \eqref{eq:2209} play an important role in the proof of the homogenization.

Now we consider the convergence of the following vector valued function defined by
\begin{equation*}
\eta_\xi^\vae
=
A^\vae\nabla w_\xi^\vae
=
\left(
\sum_{i=1}^N
a_{i1}^\vae\frac{\pp w_\xi^\vae}{\pp x_i}
,\ldots,
\sum_{i=1}^N
a_{iN}^\vae\frac{\pp w_\xi^\vae}{\pp x_i}
\right)
.
\end{equation*}
By the definition of $A^\vae$ and $w_\xi^\vae$, 
we see that 
\begin{equation*}
\eta_\xi^\vae(x)
=
\left(
A\nabla_y w_\xi
\right)
\left(
\frac{x}{\vae}
\right)
=
\frac1\vae
\left[
A
\left(
\frac{x}{\vae}
\right)
\left(
\nabla_y(\vae w_\xi)
\right)
\left(
\frac{x}{\vae}
\right)
\right].
\end{equation*}
Since $A$ is $Y$-periodic, 
$A\nabla_y w_\xi$ is $Y$-periodic
and then 
$\eta_\xi^\vae$ is $Y$-periodic. 
By Lemma \ref{lem:2201}, we have
\begin{equation}
\label{eq:2210}
\eta_\xi^\vae
\wto
\M_Y(A\nabla w_\xi)
=A^0\xi
\ \utext{weakly in $(L^2(\Omega))^N$},
\end{equation}
where $A^0$ is the matrix defined by \eqref{eq:3106} in Theorem \ref{thm:3101}.

Moreover we will show that $\eta_\xi^\vae$ satisfies
\begin{equation}
\label{eq:2211}
\int_\Omega
\eta_\xi^\vae(x)
\cdot
\nabla\Phi(x)
\,dx
=0
\quad
\utext{for $\Phi\in H_0^1(\Omega)$}
.
\end{equation}
Let $\varphi\in\D(\Omega)$ and set
\begin{equation*}
\varphi^\vae(y)
=
\varphi(\vae y)
\ \utext{a.e. on $\mathbb{R}^N$}.
\end{equation*}
Obviously, we see that $\varphi^\vae\in\D(\mathbb{R}^N)$. 
Hence we have
\begin{equation*}
\int_{\mathbb{R}^N}
(A\nabla_y w_\xi)(y)
\cdot 
\nabla_y \varphi^\vae(y)
\,dy
=0.
\end{equation*}
Since $\supp\varphi\subset\Omega$, we have
\begin{equation*}
\int_\Omega
(A\nabla_y w_\xi)\left(\frac{x}{\vae}\right)
\cdot\nabla \varphi(x)
\,dx
=0
\quad
\utext{for any $\varphi\in\D(\Omega)$}
\end{equation*}
by the change of variables $x=\vae y$. 
By the density argument, we conclude \eqref{eq:2211}.

At the end of this section, 
we see a lemma on the matrix $\widehat{A}$ defined by $w_\xi$. 
\begin{lem}
\label{lem:2202}
Let $A\in M_S(\nu,\mu,Y)$. 
We define the $N\times N$ constant matrix 
$\widehat{A}=(\widehat{a}_{ij})_{i,j=1,\ldots,N}$ by
\begin{equation*}
\widehat{A}\xi
=
\M_Y(A\nabla_y w_\xi),\quad \xi\in\mathbb{R}^N
.
\end{equation*}
Then we have 
\begin{equation*}
\widehat{A}
\in
M_S(\nu,\mu,\Omega)
.
\end{equation*}
\end{lem}
On the proof of this lemma, 
we may refer to the monograph \cite{CD1999} 
(see e.g, 
Corollary 6.10, 
Proposition 8.3
and 
Corollary 8.13
in \cite{CD1999}%
).

%
%

%
%
\section{Main Results}
First, we see the homogenization result in \S3.1. 
Then, we consider the inverse problems of determining 
the constant diffusion coefficient
for 
the time-fractional diffusion equation in \S3.2. 
Finally,
we state our inverse problems 
between different structures in \S3.3. 
%
%
\subsection{On the Homogenization}

We begin with the homogenization result for 
the time-fractional diffusion equation. 
We state the main theorem as Theorem \ref{thm:3101} in \S3.1.1. 
Then 
we see the examples of 
the homogenized coefficients 
which are coefficients of the model obtained from the homogenization. 
In Theorem \ref{thm:3101}, 
we will see that the homogenized coefficient matrix 
is given by the equality \eqref{eq:3106}. 
However, it is difficult to study this matrix directly, 
so we consider here examples that makes it easy to investigate the inverse problem.
As examples, we consider 
the one-dimensional case in space in \S3.1.2 and 
the case of the layered material in \S3.1.3.

\subsubsection{Homogenization for Fractional Diffusion Equations}
We have the following theorem. 
\begin{thm}
\label{thm:3101}
Let $f^\vae\in L^2(0,T;H^{-1}(\Omega))$ 
and $u_0^\vae\in L^2(\Omega)$. 
We assume that 
$u^\vae$ is the weak solution of the following problem in the periodic structure:
\begin{equation}
\label{eq:3101}
\left\{
\begin{aligned}
&
\pp_t^\alpha u^\vae(x,t)
-
\dd (A^\vae(x)\nabla u^\vae(x,t))
=
f^\vae(x,t),
&(x,t)\in\Omega\times(0,T)
,\\&
u^\vae(x,t)=0,
&(x,t)\in\pp\Omega\times(0,T)
,\\&
u^\vae(x,0)=u_0^\vae(x),
&x\in\Omega
,
\end{aligned}
\right.
\end{equation}
where $A^\vae$ is the coefficient matrix 
defined by \eqref{eq:2201} and \eqref{eq:2202}. 
We also assume that the following convergences:
there exists $u_0\in L^2(\Omega)$ and 
$f\in L^2(0,T;H^{-1}(\Omega))$ such that
\begin{equation}
\label{eq:3102}
\left\{
\begin{aligned}
&\utext{i)}
&&
u_0^\vae
\wto
u_0
\ \utext{weakly in $L^2(\Omega)$},
\\
&\utext{ii)}
&&
f^\vae
\to
f
\ \utext{strongly in $L^2(0,T;H^{-1}(\Omega))$}.
\end{aligned}
\right.
\end{equation}
Then $u^\vae$ satisfies the followings:
\begin{equation}
\label{eq:3103}
\left\{
\begin{aligned}
&\utext{i)}
&&
u^\vae
\wto
u^0
\ \utext{weakly in $L^2(0,T;H_0^1(\Omega))$},
\\
&\utext{ii)}
&&
\pp_t^\alpha(u^\vae-u_0^\vae)
\wto
\pp_t^\alpha(u^0-u_0)
\ \utext{weakly in $L^2(0,T;H^{-1}(\Omega))$},
\\
&\utext{iii)}
&&
A^\vae \nabla u^\vae
\wto
A^0 \nabla u^0
\ \utext{weakly in $(L^2(0,T;L^2(\Omega)))^N$}.
\end{aligned}
\right.
\end{equation}
Furthermore,
\begin{equation}
\label{eq:3104}
u^\vae\to u^0
\ \utext{strongly in $L^2(0,T;L^2(\Omega))$},
\end{equation}
where $u^0$ 
is the weak solution of the homogenized problem:
\begin{equation}
\label{eq:3105}
\left\{
\begin{aligned}
&
\pp_t^\alpha u^0(x,t)
-
\dd (A^0\nabla u^0(x,t))
=
f(x,t),
&(x,t)\in\Omega\times(0,T)
,\\&
u^0(x,t)=0,
&(x,t)\in\pp\Omega\times(0,T)
,\\&
u^0(x,0)=u_0(x),
&x\in\Omega
.
\end{aligned}
\right.
\end{equation}
Here $A^0$ is the homogenized coefficient matrix 
defined by
\begin{equation}
\label{eq:3106}
A^0\xi
=
\M_Y(A\nabla_y w_\xi),\quad \xi\in\mathbb{R}^N.
\end{equation}
\end{thm}
We will prove this theorem in \S4.1.

\begin{rmk}
\label{rmk:3101}
By Lemma \ref{lem:2202}, 
$A^0\in M_S(\nu,\mu,\Omega)$. 
Hence we see that 
there exists the unique weak solution $u^0$ of 
the homogenized problem \eqref{eq:3105} 
by Theorem \ref{thm:well-posedness_FDE}. 
\end{rmk}
\begin{rmk}
\label{rmk:3102}
We may replace the assumption for $f^\vae$ 
by the followings:
$f^\vae\in L^2(0,T;L^2(\Omega))$ and 
there exists 
$f\in L^2(0,T;L^2(\Omega))$ such that
\begin{equation*}
f^\vae
\to
f
\ \utext{strongly in $L^2(0,T;L^2(\Omega))$}.
\end{equation*}
Indeed, 
since 
the assumption for $f^\vae$ is used in the proof 
when taking limits, 
so the claim of our theorem holds under either assumption.
\end{rmk}
\begin{rmk}
On the methodology of the homogenization with periodic structures, 
we may also consider the two-scale convergence method introduced by
Nguetseng \cite{N1989} and Allaire \cite{A1992}. 
Furthermore, 
we may discuss the non-periodic cases by Theorem \ref{thm:3101}, 
but we do not deal with them here.
\end{rmk}

%
%
\subsubsection{One-Dimensional Case in Space}
We consider the fractional diffusion equation 
for one-dimensional case in space, that is, $N=1$. 
Let $Y=(0,\ell_1)$ and 
$\Omega$ be an open interval in $\mathbb{R}$. 
We denote by $|Y|$ the measure of $Y$, 
that is, $|Y|=\ell_1$. 
We assume that $a^\vae$ is defined by
\begin{equation*}
a^\vae(x)
=
a\left(\frac{x}{\vae}\right)
,\quad
x\in\mathbb{R},
\end{equation*}
where
\begin{equation*}
a\in L^\infty(Y),\quad
\nu\leq a(y)\leq \mu,\ y\in\overline{Y},\quad
\text{$a$ is $Y$-periodic},
\end{equation*}
with 
constants $\nu,\mu\in\mathbb{R}$ satisfying $0<\nu<\mu$.  

Let 
$f^\vae\in L^2(0,T;H^{-1}(\Omega))$ and
$u_0^\vae \in L^2(\Omega)$ satisfy 
convergences \eqref{eq:3102} 
with 
$f\in L^2(0,T;H^{-1}(\Omega))$ and
$u_0 \in L^2(\Omega)$. 
We assume that 
$u^\vae$ satisfies the problem:
\begin{equation*}
\left\{
\begin{aligned}
&
\pp_t^\alpha u^\vae(x,t)
-
\pp_x (a^\vae(x)\pp_x u^\vae(x,t))
=
f^\vae(x,t),
&(x,t)\in\Omega\times(0,T)
,\\&
u^\vae(x,t)=0,
&(x,t)\in\pp\Omega\times(0,T)
,\\&
u^\vae(x,0)=u_0^\vae(x),
&x\in\Omega
.
\end{aligned}
\right.
\end{equation*}
By Theorem \ref{thm:3101}, 
$u^\vae$ converges to
the weak solution $u^0$ of the following problem in the sense of the convergences \eqref{eq:3103} and \eqref{eq:3104}:
\begin{equation*}
\left\{
\begin{aligned}
&
\pp_t^\alpha u^0(x,t)
-
\pp_x (a^0\pp_x u^0(x,t))
=
f(x,t),
&(x,t)\in\Omega\times(0,T)
,\\&
u^0(x,t)=0,
&(x,t)\in\pp\Omega\times(0,T)
,\\&
u^0(x,0)=u_0(x),
&x\in\Omega
,
\end{aligned}
\right.
\end{equation*}
where 
$a^0$ is the homogenized coefficient 
\begin{equation*}
a^0\xi
=
\M_Y(a\pp_y w_\xi)
,\quad
\xi\in\mathbb{R}.
\end{equation*}
In the one-dimensional case in space, 
we may represent the homogenized coefficient $a^0$ by $a$. 
\begin{lem}
\label{lem:3101}
We have the following representation formula of $a^0$ by $a$:
\begin{equation*}
a^0
=
\frac{1}{\M_Y\left(\frac{1}{a}\right)}
.
\end{equation*}
\end{lem}
The proof of this lemma is given in Appendix.
%
%
\subsubsection{Layered Material Case}
We consider the anomalous diffusion phenomena in 
the layered material. 
Then the the diffusion coefficient in the periodic structure $A$ is written by 
$A(y)=A(y_1)$ for $y=(y_1,\ldots, y_N)\in Y$, 
where $A$ satisfies the assumptions \eqref{eq:2201}. 
In this case, we note that $A$ is $(0,\ell_1)$-periodic.  

Under the assumptions of Theorem \ref{thm:3101}, 
we may obtain the result on the homogenization 
and $A^0$ can be represented by $A$ explicitly.  
That is, we have the following lemma, which proof is given in Appendix. 
\begin{lem}
\label{lem:3102}
We have the representation formula 
$A^0=(a_{ij}^0)$ by $A(y)=A(y_1)=(a_{ij}(y_1))$:
\begin{equation*}
A^0=A^\ast, 
\end{equation*}
where $A^\ast=(a_{ij}^\ast)$ is defined by 
\begin{align*}
a_{11}^\ast
&=\frac{1}{\M_{(0,\ell_1)}\left(\frac1{a_{11}}\right)}
,\\
a_{1j}^\ast
&=a_{11}^\ast\M_{(0,\ell_1)}\left(\frac{a_{1j}}{a_{11}}\right)
,\quad
a_{i1}^\ast
=a_{11}^\ast\M_{(0,\ell_1)}\left(\frac{a_{i1}}{a_{11}}\right)
,\\
a_{ij}^\ast
&=a_{11}^\ast
\M_{(0,\ell_1)}\left(\frac{a_{1j}}{a_{11}}\right)
\M_{(0,\ell_1)}\left(\frac{a_{i1}}{a_{11}}\right)
+
\M_{(0,\ell_1)}\left(a_{ij}-\frac{a_{1j}a_{i1}}{a_{11}}\right)
\end{align*}
for $i,j=2,\ldots,N$.
\end{lem}
Since $A$ is the symmetric matrix,  
we may see that $A^\ast$ in 
Lemma \ref{lem:3102} is the symmetric matrix, too. 

In particular, 
if the coefficient matrix $A$ is the diagonal matrix, 
we have the following result as a direct consequence of the above lemma. 
\begin{lem}
\label{lem:3103}
If $A$ is defined by 
\begin{equation*}
A(y)
=
\begin{pmatrix}
a_1(y_1)	&0			&\cdots	&0\\
0		&a_2(y_1)	&		&\vdots\\
\vdots	&			&\ddots	&0\\
0		&\cdots		&0		&a_N(y_1)
\end{pmatrix}
,\quad
y=(y_1,\ldots,y_N)\in Y,
\end{equation*}
then
we may represent $A^0=(a_{ij}^0)_{i,j=1,\ldots,N}$ by 
$a_1(y_1),a_2(y_1),\ldots,a_N(y_1)$:
\begin{align*}
a_{11}^0
&=\frac{1}{\M_{(0,\ell_1)}\left(\frac1{a_1}\right)}
,&\\
a_{ii}^0
&=\M_{(0,\ell_1)}\left(a_i\right)
,&i=2,\ldots,N,\\
a_{ij}^0
&=0
,&i\neq j.
\end{align*}
\end{lem}

%
%
\subsection{Determination of Constant Diffusion Coefficient}

We discuss the inverse coefficient problems for the homogenized time-fractional diffusion equations. 
We first consider the inverse problem when the coefficient is a scalar in \S3.2.1. 
Then we consider the determination of the coefficient in the case of the layered material in \S3.2.2.

%
%
\subsubsection{The case where the diffusion coefficient is a scaler}

Let $u_p$ satisfy 
the problem: 
\begin{equation}
\label{eq:32101}
\left\{
\begin{aligned}
&
\pp_t^\alpha u_p(x,t)
-
p \Delta u_p(x,t)
=
0,
&(x,t)\in\Omega\times(0,T)
,\\&
u_p(x,t)=0,
&(x,t)\in\pp\Omega\times(0,T)
,\\&
u_p(x,0)=u_0(x),
&x\in\Omega
,
\end{aligned}
\right.
\end{equation}
where $p$ is a constant satisfying 
$0<\nu\leq p\leq \mu$ with some constants $\nu,\mu$.

Let $\mathcal{A}:=\mathcal{A}_1:=-\Delta$ be 
an operator in $L^2(\Omega)$ with the domain $\D(\mathcal{A})=H^2(\Omega)\cap H_0^1(\Omega)$. 
Let 
\begin{equation*}
0<\la_1 < \la_2 < \la_3 < \cdots \longrightarrow  \infty
\end{equation*}
be the set of the eigenvalues of $\mathcal{A}$ and let $\{\vap_{n_j}\}_{j =1,\ldots d_n}$ be
an orthonormal basis of $\Ker(\mathcal{A}-\la_n)$, 
$P_n: L^2(\Omega) \to\Ker(\mathcal{A}-\la_n)$ be the eigenprojection:
\begin{equation*}
P_n v = \sum_{j=1}^{d_n} (v,\vap_{n_j})_{L^2(\Omega)}\vap_{n_j},\quad v\in L^2(\Omega)
.
\end{equation*}
Throughout this subsection and the proofs of the following theorems, 
we denote the eigenfunctions by $\varphi_{n_j}$. 

Now we investigate the following inverse coefficient problem. 

\noindent
\textbf{Inverse Problem 1}: 
Let $x_0\in\Omega$ and $t_0\in (0,T)$ be arbitrarily chosen. 
Determine the constant diffusion coefficient $p$ 
by the data $u_p(x_0,t_0)$.

The key point of solving this inverse problem is 
to use the minimum observation data. 
We have the following theorem on the stability result for our inverse problem. We will prove this theorem in \S 4.2.1.

\begin{thm}[Stability]
\label{thm:32101}
We assume that $u_0\in H^2(\Omega) \cap H^1_0(\Omega)$ satisfies 
\begin{equation}
\label{eq:32102}
u_0\not\equiv 0 \quad \text{\rm in $\Omega$} \quad \text{\rm and}\quad 
\Delta u_0(x) \ge 0 \quad \text{\rm a.e. $x\in \Omega$}.
\end{equation}
Then there exists a constant
$C=C(\nu,\mu) > 0$ such that 
\begin{equation*}
|p-q| \leq C | u_p(x_0,t_0) - u_q(x_0,t_0)|.
\end{equation*}
\end{thm}

Next we consider the following inverse problem. 

\noindent
\textbf{Inverse Problem 2}: 
Let us choose 
a sub-domain $\omega\subset\Omega$ and 
an open interval $I\subset (0,T)$. 
Determine the constant diffusion coefficient 
$p$ 
by the data
\begin{equation*}
\int_I\int_\omega u_p(x,t)\,dxdt.
\end{equation*}

Then we have the following stability result on this inverse problem. 
\begin{thm}[Stability]
\label{thm:32104}
We assume that $u_0\in H^2(\Omega) \cap H^1_0(\Omega)$ satisfies \eqref{eq:32102}
Then there exists a constant
$C=C(\nu,\mu) > 0$ such that 
\begin{equation*}
|p-q| 
\leq 
C 
\left|
\int_I\int_\omega u_p(x,t)\,dxdt
 - 
\int_I\int_\omega u_q(x,t)\,dxdt
\right|
.
\end{equation*}
\end{thm}
We will show this theorem in \S4.2.2.

\begin{rmk}
In 
Theorem \ref{thm:32101} and Theorem \ref{thm:32104}, we can replace the latter part of the assumptions \eqref{eq:32102} by $\Delta u_0(x) \le 0$ a.e in $x\in\Omega$.
\end{rmk}

\begin{rmk}
We can not extend the results of 
Theorem \ref{thm:32101} and Theorem \ref{thm:32104}
to the case of $1<\alpha<2$, because the key is
the maximum principle.
\end{rmk}

Furthermore, we may prove the uniqueness by stronger data. 
Although we deal with the constant as the observation data in the previous two inverse problems, 
we consider the time-dependent functions as the observation data in the following inverse problem. 

\noindent
\textbf{Inverse Problem 3}: 
Let $x_0\in\Omega$ and $t_1\in (0,T)$ be arbitrarily chosen. 
Determine the constant diffusion coefficient $p$ 
by the data $u_p(x_0,t)$, $0<t<t_1$.

Related to this inverse problem, we state two theorem on the uniqueness: Theorem \ref{thm:32102} and \ref{thm:32103}, which we show in \S 4.2.3.

\begin{thm}[Uniqueness]
\label{thm:32102}
Let $u_0 \in H^2(\Omega) \cap H^1_0(\Omega)$.
We assume that there exists $n_0\in \mathbb{N}$ such that 
\begin{equation}
\label{eq:32103}
P_{n_0}u_0(x_0) \ne 0.
\end{equation}
If $u_p(x_0,t) = u_q(x_0,t)$ for $0<t<t_1$,
then we have $p=q$.
\end{thm}
\begin{rmk}
In particular, $u_0(x_0) \ne 0$ implies 
\eqref{eq:32103}.
\end{rmk}

We can consider the uniqueness for $p$ even for unknown initial values
$u_{0,p}$.    
Let $\widetilde{u}_{p}$ satisfy 
\begin{equation}
\label{eq:32104}
\left\{
\begin{aligned}
&
\pp_t^\alpha \widetilde{u}_{p}(x,t)
-
p \Delta \widetilde{u}_{p}(x,t)
=
0,
&(x,t)\in\Omega\times(0,T)
,\\&
\widetilde{u}_{p}(x,t)=0,
&(x,t)\in\pp\Omega\times(0,T)
,\\&
\widetilde{u}_{p}(x,0)=u_{0,p}(x),
&x\in\Omega
,
\end{aligned}
\right.
\end{equation}
\begin{thm}[Uniqueness with unknown initial values]
\label{thm:32103}
We assume that $u_{0,p} \in H^2(\Omega) \cap H^1_0(\Omega)$ satisfy
\begin{equation}
\label{eq:32105}
P_1 u_{0,p}(x_0)\ne 0\quad\text{\rm and}\quad
P_1 u_{0,q}(x_0)\ne 0
.
\end{equation}
If $\widetilde{u}_{p}(x_0,t) = \widetilde{u}_{q}(x_0,t)$ 
for $0<t<t_1$,
then we have $p=q$.
\end{thm}
The assumption \eqref{eq:32105} is essential for the uniqueness if we do not 
assume 
$u_{0,p}\neq u_{0,q}$ 
in $\Omega$, as the following example shows.

\begin{exa}
Ler $\Omega= (0, \pi)$.  Then $P_k u_{0,p} = (u_{0,p},\vap_k)\vap_k$ for $k\in \mathbb{N}$ 
and 
$\vap_k(x) = \sqrt{\frac{2}{\pi}}\sin kx$ and 
$\la_k = k^2$ for $k\in \mathbb{N}$.
Let $x_0=\frac{\pi}{3}$, $p=1$, $q= \frac{1}{4}$,
$u_{0,p}(x)= \vap_1(x)$ and $u_{0,q}(x) = \vap_2(x)$.
Then $\widetilde{u}_{p}(x,t) = E_{\alpha,1}(-t^{\alpha})\vap_1(x)$ and
$\widetilde{u}_{q}(x,t) = E_{\alpha,1}(-t^{\alpha})\vap_2(x)$.
We can readily see that $\widetilde{u}_{p}(x_0,t) = \widetilde{u}_{q}(x_0,t)$ for
$t>0$. 
In this case, we note that $P_1 u_{0,q}=0$ in $\Omega$, 
that is, 
\eqref{eq:32105} is not satisfied with $u_{0,q}$.
\end{exa}

\begin{rmk}
We may extend Theorem \ref{thm:32102} and Theorem \ref{thm:32103} 
to the case of $1<\alpha<2$. 
Furthermore we can prove that 
$u_k(x_0,t_0)$ also simultaneously determine 
the order $\alpha\in (0,2)$ (see e.g., \cite{Y2021}).  
\end{rmk}

\begin{rmk}
Using the uniqueness results (Theorem \ref{thm:32102} and Theorem \ref{thm:32103}), 
we may discuss the stability reslts 
by the Tikhonov Theorem (see e.g., \cite{I2006}). 
But we omit them here. 
\end{rmk}
%
%
\subsubsection{The case of the layered material}

Let $N\geq 2$. 
Let $\Omega=(0,\delta)\times D$ be the cylindrical domain, 
where $\delta>0$ and $D$ is a bounded domain in $\mathbb{R}^{N-1}$ with $C^2$-class boundary $\pp D$. 
Denote the element of $\Omega$ by $x=(x_1,\widetilde{x})\in\Omega$, where 
$x_1\in (0,\delta)$ and 
$\widetilde{x}=(x_2,\ldots,x_N)\in D$. 

Let $u_p$ satisfy problems for $k=1,2$: 
\begin{equation}
\label{eq:32201}
\left\{
\begin{aligned}
&
\pp_t^\alpha u_p(x,t)
-
\dd(B_p \nabla u_p(x,t))
=
0,
&(x,t)\in\Omega\times(0,T)
,\\&
u_p(x,t)=0,
&(x,t)\in\pp\Omega\times(0,T)
,\\&
u_p(x,0)=u_0(x),
&x\in\Omega
.
\end{aligned}
\right.
\end{equation}
Here
$B_p$ is the diagonal matrix 
defined by
\begin{equation*}
B_p
=
\begin{pmatrix}
p		&0			&\cdots	&0\\
0		&b_2			&		&\vdots\\
\vdots	&			&\ddots	&0\\
0		&\cdots		&0		&b_N
\end{pmatrix}
,
\end{equation*}
where
$
\nu\leq p,b_i\leq \mu
$
for $i=2,\ldots,N$ 
with 
some constants $\nu,\mu\in\mathbb{R}$ satisfying $0<\nu<\mu$.  

Let $\mathcal{B}:=-\pp_1^2-b_2\pp_2^2-\cdots -b_N\pp_N^2$ be 
an operator in $L^2(\Omega)$ with the domain $\D(\mathcal{B})=H^2(\Omega)\cap H_0^1(\Omega)$. 
Let us consider the eigenvalue problem:
\begin{equation}
\label{eq:32202}
\mathcal{B}\varphi=\lambda\varphi\;\text{in $\Omega$}
,\quad
\varphi\in\D(\mathcal{B}).
\end{equation}
We will find the eigenfunction 
$\varphi(x)=\Phi(x_1)\widetilde\Phi(\widetilde{x})$ 
by the separation of the variables. 
Then we may decompose the above eigenvalue problem 
into the following two eigenvalue problems.
\begin{equation}
\label{eq:32203}
-\pp_1^2 \Phi
=
\Lambda\Phi\; \text{in $(0,\ell_1)$}
,\quad
\Phi\in H^2(0,\ell_1)\cap H_0^1(0,\ell_1)
\end{equation}
and
\begin{equation}
\label{eq:32204}
\widetilde{\mathcal{B}} \widetilde\Phi
=
(\lambda-\Lambda)\widetilde\Phi\; \text{in $D$}
,\quad
\widetilde\Phi\in H^2(D)\cap H_0^1(D),
\end{equation}
where 
$\widetilde{\mathcal{B}}
=-b_2\pp_2^2-\cdots-b_N\pp_N^2$ is an operator in $L^2(D)$ 
with the domain 
$\D(\widetilde{\mathcal{B}})=H^2(D)\cap H_0^1(D)$.
We may solve the above problems and we obtain 
the following eigenvalues with multiplicities
\begin{equation*}
0<
\Lambda_{1}<
\Lambda_{2}\leq
\cdots\leq
\Lambda_{n}
\leq \cdots\to \infty
,\quad
0<
\widetilde\Lambda_{1}<
\widetilde\Lambda_{2}\leq
\cdots\leq
\widetilde\Lambda_{m}
\leq \cdots\to \infty
\end{equation*}
with $\lambda_{n,m}=\Lambda_n+\widetilde{\Lambda}_m$
and corresponding eigenfunctions 
$\Phi_n$, $\widetilde{\Phi}_m$ satisfying
\begin{equation*}
(\Phi_k,\Phi_\ell)_{L^2(0,\ell_1)}=\delta_{k,\ell},\quad
(\widetilde\Phi_k,\widetilde\Phi_\ell)_{L^2(D)}=\delta_{k,\ell},
\end{equation*}
where $\delta_{k,\ell}$ is the Kronecker delta for $k,\ell\in\mathbb{N}$.
Thus we obtain the eigenvalues and eigenfunctions for the eigenvalue problem \eqref{eq:32202}:
\begin{equation*}
\lambda_{n,m}=\Lambda_n+\widetilde{\Lambda}_m,\quad
\varphi_{n,m}(x)=\Phi_n(x_1)\widetilde\Phi_m(\widetilde{x}).
\end{equation*}
We also see that the eigenvalues are only these above.
Indeed, let us consider the problem \eqref{eq:32202} 
for arbitrarily fixed eigenvalue $\lambda$. 
Then the Fourier series expansion for $\varphi$ with respect to $x_1$ and $\widetilde{x}$ yields
\begin{equation*}
\varphi(x)
=\sum_{n,m=1}^\infty c_{nm}
\Phi_n(x_1)\widetilde\Phi_m(\widetilde{x})
\end{equation*}
with some constant 
$c_{nm}=(\varphi,\Phi_n\widetilde\Phi_m)_{L^2(\Omega)}$. 
Then we may choose at least one pair of $n,m$ such that $c_{nm}\neq 0$. 
Multiplying the equation in \eqref{eq:32202} by 
$\Phi_n\widetilde\Phi_m$ and integrating over $\Omega$, 
we may obtain $\lambda=\Lambda_n+\widetilde{\Lambda}_m$.

For simplicity, we rearrange the eigenvalues $\lambda_{n,m}$ 
in increasing order without multiplicities.
Let
\begin{equation*}
0<\lambda_1<\lambda_2<\lambda_3<\cdots\to\infty
\end{equation*}
be the set of eigenvalues of $\mathcal{B}$ 
and let $\{\vap_{n_j}\}_{j =1,\ldots d_n}$ be an orthonormal basis 
of $\Ker (\mathcal{B}-\lambda_n)$, 
$P_n:L^2(\Omega)\to \Ker (\mathcal{B}-\lambda_n)$ 
be the eigenprojection
\begin{equation*}
P_n v
=\sum_{j=1}^{d_n}
(v,\varphi_{n_j})_{L^2(\Omega)}
\varphi_{n_j}
,\quad 
v\in L^2(\Omega). 
\end{equation*}
Then we may decompose $\lambda_n$ and $\varphi_{n_j}$ into the followings:
\begin{equation*}
\lambda_{n}
=\Lambda_{n_j}+\widetilde\Lambda_{n_j}
\quad\text{and}\quad
\varphi_{n_j}
=\Phi_{n_j}\widetilde\Phi_{n_j},
\end{equation*}
where
$\Lambda_{n_j},\widetilde\Lambda_{n_j}$
and
$\Phi_{n_j},\widetilde\Phi_{n_j}$
are eigenvalues and eigenfunctions of the problems
\eqref{eq:32203}, \eqref{eq:32204} satisfying
\begin{equation*}
\Lambda_{n_1}\leq
\Lambda_{n_2}\leq\cdots\leq
\Lambda_{n_{d_n}},\quad
\widetilde\Lambda_{n_1}\geq
\widetilde\Lambda_{n_2}\geq\cdots\geq
\widetilde\Lambda_{n_{d_n}}.
\end{equation*}
We observe that 
the first eigenvalue $\lambda_1$ is simple: $d_1=1$ and 
$\lambda_1$ is written by the sum of 
the first eigenvalues $\Lambda_1$, $\widetilde{\Lambda}_1$ 
of the problems \eqref{eq:32203} and \eqref{eq:32204}:
\begin{equation*}
\lambda_1=\Lambda_1+\widetilde{\Lambda}_1, 
\end{equation*}
where  $\Lambda_1$, $\widetilde{\Lambda}_1$ are simple, too 
(see e.g., Theorem 8.38 (p.214) in Gilbarg and Trudinger \cite{GT}). 

Now we consider the following inverse coefficient problem. 

\noindent
\textbf{Inverse Problem 1}: 
Let $x_0\in\Omega$ and $t_0\in (0,T)$ be arbitrarily chosen. 
Determine the constant diffusion coefficient $p$ 
by the data $u_p(x_0,t_0)$.

We obtain the similar results on this inverse problem 
for the case of the layered material 
as we do for the case of a scalar coefficient. 
We will show this theorem in \S 4.2.4.

\begin{thm}[Stability]
\label{thm:32201}
We assume that $u_0\in H^2(\Omega) \cap H^1_0(\Omega)$ satisfies 
\begin{equation}
\label{eq:32205}
u_0\not\equiv 0 \quad \text{\rm in $\Omega$},\quad 
\pp_1^2 u_0(x) \ge 0\;\text{\rm and}\; 
\Delta u_0(x) \ge 0 \quad \text{\rm a.e. $x\in \Omega$}.
\end{equation}
Then there exists a constant
$C=C(\nu,\mu) > 0$ such that 
\begin{equation*}
|p-q| \leq C | u_p(x_0,t_0) - u_q(x_0,t_0)|.
\end{equation*}
\end{thm}

Let us
investigate the following inverse problem. 

\noindent
\textbf{Inverse Problem 2}: 
Let us choose 
a sub-domain $\omega\subset\Omega$ and 
an open interval $I\subset (0,T)$. 
Determine the constant diffusion coefficient 
$p$ 
by the data
\begin{equation*}
\int_I\int_\omega u_p(x,t)\,dxdt.
\end{equation*}

Then we have the following theorem. 
\begin{thm}[Stability]
\label{thm:32204}
We assume that $u_0\in H^2(\Omega) \cap H^1_0(\Omega)$ satisfies \eqref{eq:32205}
Then there exists a constant
$C=C(\nu,\mu) > 0$ such that 
\begin{equation*}
|p-q| 
\leq 
C 
\left|
\int_I\int_\omega u_p(x,t)\,dxdt
 - 
\int_I\int_\omega u_q(x,t)\,dxdt
\right|
.
\end{equation*}
\end{thm}
We will prove this theorem in \S4.2.5.

\begin{rmk}
In Theorem \ref{thm:32201} and Theorem \ref{thm:32204}, 
we can replace the latter part of the assumptions \eqref{eq:32205} by $\pp_1^2 u_0(x) \le 0\;\text{\rm and}\; \Delta u_0(x) \le 0$ a.e in $x\in\Omega$.
\end{rmk}

Although we cannot prove the uniqueness such as Theorem \ref{thm:32102}, 
we may consider the uniqueness for $p$ for unknown initial values
$u_{0,p}$ for the layered material. 

Let $\widetilde{u}_{p}$ satisfy 
\begin{equation}
\label{eq:32206}
\left\{
\begin{aligned}
&
\pp_t^\alpha \widetilde{u}_{p}(x,t)
-
\dd(B_p \nabla u_p(x,t))
=
0,
&(x,t)\in\Omega\times(0,T)
,\\&
\widetilde{u}_{p}(x,t)=0,
&(x,t)\in\pp\Omega\times(0,T)
,\\&
\widetilde{u}_{p}(x,0)=u_{0,p}(x),
&x\in\Omega
.
\end{aligned}
\right.
\end{equation}

\noindent
\textbf{Inverse Problem 3}: 
Let $x_0\in\Omega$ and $t_1\in (0,T)$ be arbitrarily chosen. 
Determine the constant diffusion coefficient $p$ 
by the data $u_p(x_0,t)$, $0<t<t_1$.

\begin{thm}[Uniqueness with unknown initial values]
\label{thm:32203}
Let $p,q\geq 1$. 
We assume that $u_{0,p} \in H^2(\Omega) \cap H^1_0(\Omega)$ satisfy
\begin{equation}
\label{eq:32207}
P_1 u_{0,p}(x_0)\ne 0\quad\text{\rm and}\quad
P_1 u_{0,q}(x_0)\ne 0
.
\end{equation}
If $\widetilde{u}_{p}(x_0,t) = \widetilde{u}_{q}(x_0,t)$ 
for $0<t<t_1$ with some constant $t_1$,
then we have $p=q$.
\end{thm}
We will prove Theorem \ref{thm:32203} in 
\S 4.2.6.

%
%
\subsection{Inverse Problems between Different Structures}

We discuss inverse coefficient problems 
to the two examples considered in \S3.1.2 and \S3.1.3. 
That is, we investigate 
the inverse coefficient problems 
for the one-dimensional case in space in \S3.3.1
and the layered material case in \S3.3.2.
We consider the following two types of inverse problems.
\begingroup
\leftmargini=2em
\begin{itemize}
\item
\textbf{Inverse Problem I}: 
Determine the diffusion coefficient 
for the homogenized equations of the problem \eqref{eq:3105} from the data of the problem \eqref{eq:3101} in the periodic structure.  
\item
\textbf{Inverse Problem II}: 
Determine the diffusion coefficient 
for the periodic equations of the problem \eqref{eq:3101} from the data of the problem \eqref{eq:3105}  in the homogenized structure.
\end{itemize}
\endgroup

%
%
\subsubsection{One-Dimensional Case in Space}

Let $Y=(0,\ell_1)$ and 
$\Omega$ be an open interval in $\mathbb{R}$. 
We denote by $|Y|$ the measure of $Y$. 
We assume that $a_k^\vae$ is defined by
\begin{equation*}
a_k^\vae(x)
=
a_k\left(\frac{x}{\vae}\right)
,\quad
x\in\mathbb{R}
,\quad
k=1,2,
\end{equation*}
where
\begin{equation*}
a_k\in L^\infty(Y),\quad
\nu\leq a_k(y)\leq \mu,\ y\in \overline{Y},\quad
\text{$a_k$ is $Y$-periodic}
,\quad
k=1,2
\end{equation*}
with 
constants $\nu,\mu\in\mathbb{R}$ satisfying $0<\nu<\mu$.

Let 
$u_{0,k}^\vae \in L^2(\Omega)$ satisfy 
convergence \eqref{eq:3102} 
with 
$u_{0,k} \in L^2(\Omega)$. 
We assume that 
$u_k^\vae$ satisfies the problem:
\begin{equation}
\label{eq:33102}
\left\{
\begin{aligned}
&
\pp_t^\alpha u_k^\vae(x,t)
-
\pp_x (a_k^\vae(x)\pp_x u_k^\vae(x,t))
=
0,
&(x,t)\in\Omega\times(0,T)
,\\&
u_k^\vae(x,t)=0,
&(x,t)\in\pp\Omega\times(0,T)
,\\&
u_k^\vae(x,0)=u_{0,k}^\vae(x),
&x\in\Omega
.
\end{aligned}
\right.
\end{equation}


By Theorem \ref{thm:3101} and Lemma \ref{lem:3101}, 
we may see that 
$u_k^\vae$ converges to
the weak solution $u_k^0$ of the following problem in the sense of the convergences \eqref{eq:3103} and \eqref{eq:3104}:
\begin{equation}
\label{eq:33103}
\left\{
\begin{aligned}
&
\pp_t^\alpha u_k^0(x,t)
-
\pp_x (a_k^0\pp_x u_k^0(x,t))
=
0,
&(x,t)\in\Omega\times(0,T)
,\\&
u_k^0(x,t)=0,
&(x,t)\in\pp\Omega\times(0,T)
,\\&
u_k^0(x,0)=u_{0,k}(x),
&x\in\Omega
,
\end{aligned}
\right.
\end{equation}
where 
$a_k^0$ is the homogenized coefficient defined by
\begin{equation*}
a_k^0
= 
\frac{1}{\M_Y\left(\frac{1}{a_k}\right)}
.
\end{equation*}

Let us consider the following inverse problem of type I. 

\noindent
\textbf{Inverse Problem I}:
Let us choose 
a sub-domain $\omega\subset\Omega$ and 
an open interval $I\subset (0,T)$. 
Determine the constant diffusion coefficient $a_k^0$ of the problem \eqref{eq:33103} from the data 
of the problem \eqref{eq:33102}
\begin{equation*}
\int_I\int_\omega u_k^\vae(x,t)\,dxdt.
\end{equation*}

We can obtain the asymptotic stability results for 
the above inverse problem 
as a direct consequence of Theorem \ref{thm:3101} 
and the results on the inverse problem of determining the constant diffusion coefficient: Theorem \ref{thm:32104} in \S3.2. 

In the following theorem, 
we assume that 
\eqref{eq:33103} is derived from 
\eqref{eq:33102} by the homogenization. 

\begin{thm}[Asymptotic Stability]
\label{thm:33100}
Let $u_0:=u_{0,1}\equiv  u_{0,2}$ in $\Omega$. 
We assume that $u_0\in H^2(\Omega) \cap H^1_0(\Omega)$ satisfies 
\begin{equation*}
u_0\not\equiv 0 \quad \text{\rm in $\Omega$} \quad \text{\rm and}\quad 
\Delta u_0(x) \ge 0 \quad \text{\rm a.e. $x\in \Omega$}.
\end{equation*}
Then there exists a constant
$C=C(\nu,\mu) > 0$ such that 
\begin{equation*}
\left|
a_1^0-a_2^0
\right|
\leq 
C \left| 
\int_I\int_\omega u_1^\vae(x,t)\,dxdt
-
\int_I\int_\omega u_2^\vae(x,t)\,dxdt
\right|
+
\theta(\vae)
\end{equation*}
for all $\vae>0$. 
Here
$\theta(\vae)\to 0$ as $\vae\to 0$.
\end{thm}
\begin{proof}
By the triangle inequality,  we have
\begin{align*}
\left|
a_1^0-a_2^0
\right| 
&\leq 
C 
\left|
\int_I\int_\omega u_1^0(x,t)\,dxdt
 - 
\int_I\int_\omega u_2^0(x,t)\,dxdt
\right|
\\&=
C 
\left|
\int_I\int_\omega \left(u_1^0(x,t)-u_2^0(x,t)\right)\,dxdt
\right|
\\&\leq
C
\left|
\int_I\int_\omega \left(u_1^0(x,t)-u_1^\vae(x,t)\right)\,dxdt
\right|
\\&\quad
+C
\left|
\int_I\int_\omega \left(u_2^\vae(x,t)-u_2^0(x,t)\right)\,dxdt
\right|
\\&\quad
+C
\left|
\int_I\int_\omega \left(u_1^\vae(x,t)-u_2^\vae(x,t)\right)\,dxdt
\right|
\\&\leq
C \left| 
\int_I\int_\omega u_1^\vae(x,t)\,dxdt
-
\int_I\int_\omega u_2^\vae(x,t)\,dxdt
\right|
\\&\quad
+
C
\left(
\left\|u_1^\vae-u_1^0\right\|_{L^2(0,T;L^2(\Omega))}
+
\left\|u_2^\vae-u_2^0\right\|_{L^2(0,T;L^2(\Omega))}
\right).
\end{align*}
Set
\begin{equation*}
\theta(\vae)
=
C
\left(
\left\|u_1^\vae-u_1^0\right\|_{L^2(0,T;L^2(\Omega))}
+
\left\|u_2^\vae-u_2^0\right\|_{L^2(0,T;L^2(\Omega))}
\right).
\end{equation*}
Then we see that $\theta(\vae)\to 0$ as $\vae\to 0$
by the convergence \eqref{eq:3104} of Theorem \ref{thm:32104}.

Thus we conclude the proof of Theorem \ref{thm:33100}.
\end{proof}
\begin{rmk}
By the direct calculation,  
we get
\begin{equation}
\label{eq:33100}
\left|
a_1^0-a_2^0
\right|
\leq
C
\left\|
a_1-a_2
\right\|_{L^1(Y)}.
\end{equation}
Hence, we may obtain the stability result by using 
stronger observation data. 
\end{rmk}

Next we consider the following inverse problem of type II. 

\noindent
\textbf{Inverse Problem II-1}:
Let $x_0\in\Omega$ and $t_0\in (0,T)$ be arbitrarily chosen. 
Determine the diffusion coefficient $a_k$ of the problem \eqref{eq:33102} from the data $u_k^0(x_0,t_0)$ of the problem \eqref{eq:33103}.

We start with showing 
the relation between $a_k$ and $a_k^0$ 
by a straightforward calculation.
To prove the following lemma, we assume the a priori condition:
\begin{equation}
\label{eq:33101}
a_1(y)\geq a_2(y)
,\quad
y\in \overline{Y}.
\end{equation}

In the followings, 
we assume that 
\eqref{eq:33103} is derived from 
\eqref{eq:33102} by the homogenization 
under its assumptions and \eqref{eq:33101}.

\begin{lem}
\label{lem:33101}
We have
\begin{equation*}
C^{-1}
\left\|
a_1-a_2
\right\|_{L^1(Y)}
\leq
a_1^0-a_2^0
\leq
C
\left\|
a_1-a_2
\right\|_{L^1(Y)}
\end{equation*}
with the constant $C>0$ depending on only $\nu,\mu,\ell_1$. 
\end{lem}
\begin{proof}
Since $\nu\leq a_k(y) \leq \mu$ for $y\in \overline{Y}$, 
we have 
\begin{equation*}
\frac1\mu
\leq
\frac1{a_k(y)}
\leq
\frac1\nu,\ y\in \overline{Y}
,\quad
\nu\leq
\frac1{\M_{(0,\ell_1)}\left(\frac{1}{a_k}\right)}
\leq\mu
\end{equation*}
for $k=1,2$. 
So we obtain
\begin{align*}
a_1^0-a_2^0
&=
\frac1{\M_{(0,\ell_1)}\left(\frac{1}{a_1}\right)}
-
\frac1{\M_{(0,\ell_1)}\left(\frac{1}{a_2}\right)}
=
\frac{\M_{(0,\ell_1)}\left(\frac{1}{a_2}\right)-\M_{(0,\ell_1)}\left(\frac{1}{a_1}\right)}{\M_{(0,\ell_1)}\left(\frac{1}{a_1}\right)\M_{(0,\ell_1)}\left(\frac{1}{a_2}\right)}
\\&\geq 
\frac{\nu^2}{\ell_1}
\int_0^{\ell_1}
\left(
\frac1{a_2}
-
\frac1{a_1}
\right)
\,dy
=
\frac{\nu^2}{\ell_1}
\int_0^{\ell_1}
\frac{a_1-a_2}{a_1a_2}
\,dy
\geq
\frac{\nu^2}{\mu^2\ell_1}
\int_0^{\ell_1}
(a_1-a_2)
\,dy.
\end{align*}
This with \eqref{eq:33100} yields \eqref{eq:33101}.

Thus we complete the proof of Lemma \ref{lem:33101}. 
\end{proof}
\begin{rmk}
We need the strong assumption \eqref{eq:33101} on the relation between $a_1$ and $a_2$. 
In our inverse problems, 
we consider the problem to find $a_k$ from $a_k^0$. 
This is the same problem 
as 
the problem to find a function from its average (the harmonic mean). 
So we need the a priori assumption \eqref{eq:33101}.
\end{rmk}

We can obtain the results for 
the above inverse problem 
as a direct consequence of Lemma \ref{lem:33101} 
and the results on the inverse problem of determining the constant diffusion coefficient: Theorems \ref{thm:32101} in \S3.2.

\begin{prop}[Stability]
\label{prop:33101}
Let $u_0:=u_{0,1}\equiv  u_{0,2}$ in $\Omega$. 
We assume that $u_0\in H^2(\Omega) \cap H^1_0(\Omega)$ satisfies 
\begin{equation*}
u_0\not\equiv 0 \quad \text{\rm in $\Omega$} \quad \text{\rm and}\quad 
\Delta u_0(x) \ge 0 \quad \text{\rm a.e. $x\in \Omega$}.
\end{equation*}
Then there exists a constant
$C=C(\nu,\mu,Y) > 0$ such that 
\begin{equation*}
\left\|
a_1-a_2
\right\|_{L^1(Y)}
\leq C \left| u_1^0(x_0,t_0) - u_2^0(x_0,t_0)\right|.
\end{equation*}
\end{prop}
\begin{proof}
By Theorem \ref{thm:32101}, we have
\begin{equation*}
\left|a_1^0-a_2^0\right|
\leq C \left| u_1^0(x_0,t_0) - u_2^0(x_0,t_0)\right|.
\end{equation*}
Combining this with the inequality in Lemma \ref{lem:33101}, 
we may get the stability estimate of this theorem. 
Thus the proof is complete. 
\end{proof}

Furthermore, we may consider 
another inverse problem of type II. 

\noindent
\textbf{Inverse Problem II-2}:
Let $x_0\in\Omega$ and $t_1\in (0,T)$ be arbitrarily chosen. 
Determine the diffusion coefficient $a_k$ of the problem \eqref{eq:33102} from the data $u_k^0(x_0,t)$, $0<t<t_1$ of the problem \eqref{eq:33103}.

By Theorems \ref{thm:32102}, \ref{thm:32103} and Lemma \ref{lem:33101}, 
we obtain the following two theorems on the uniqueness. 
\begin{prop}[Uniqueness]
\label{prop:33102}
Let $u_0:=u_{0,1}\equiv  u_{0,2}$ in $\Omega$
and let $u_0 \in H^2(\Omega) \cap H^1_0(\Omega)$.
We assume that there exists $n_0\in \mathbb{N}$ such that 
\begin{equation*}
P_{n_0}u_0(x_0) \ne 0,
\end{equation*}
where $P_{n_0}$ is the eigenprojection defined in \S3.2.1. 
If $u_1^0(x_0,t) = u_2^0(x_0,t)$ for $0<t<t_1$,
then we have $a_1=a_2$ a.e.\ on $Y$.
\end{prop}

\begin{prop}[Uniqueness with unknown initial values]
\label{prop:33103}
We assume that $u_{0,k} \in H^2(\Omega) \cap H^1_0(\Omega)$ satisfy
\begin{equation*}
P_1 u_{0,1}(x_0)\ne 0\quad\text{\rm and}\quad
P_1 u_{0,2}(x_0)\ne 0
,
\end{equation*}
where $P_1$ is the eigenprojection defined in \S3.2.1.
If $u_1^0(x_0,t) = u_2^0(x_0,t)$ 
for $0<t<t_1$,
then we have $a_1=a_2$ a.e.\ on $Y$.
\end{prop}
%
%
\subsubsection{Layered Material Case}
Now we consider the case of the layered material. Let $N\geq 2$. 
Let $\Omega=(0,\delta)\times D$ be the cylindrical domain, 
where $\delta>0$ and $D$ is a bounded domain in $\mathbb{R}^{N-1}$ with $C^2$-class boundary $\pp D$. 
Denote the element of $\Omega$ by $x=(x_1,\widetilde{x})\in\Omega$, where 
$x_1\in (0,\delta)$ and $\widetilde{x}=(x_2,\ldots,x_N)\in D$. 
We suppose that 
the diffusion coefficient matrix is the diagonal matrix 
depending on only one variable in space, that is, we assume that 
\begin{equation*}
A_k(y)
=
\begin{pmatrix}
p_k(y_1)	&0			&\cdots	&0\\
0		&a_2(y_1)	&		&\vdots\\
\vdots	&			&\ddots	&0\\
0		&\cdots		&0		&a_N(y_1)
\end{pmatrix}
,\quad
y=(y_1,\ldots,y_N)\in Y
\end{equation*}
satisfies 
\begin{equation*}
p_k,a_i\in L^\infty(0,\ell_1),\;
\nu\leq p_k(y_1),a_{i}(y_1)\leq \mu,\, y_1\in [0,\ell_1],\;
\text{$p_k,a_{i}$ are $(0,\ell_1)$-periodic}
\end{equation*}
for 
$k=1,2$,  
$i=2,\ldots,N$ 
with 
constants $\nu,\mu\in\mathbb{R}$ satisfying $0<\nu<\mu$.  
Hence $A_k$ satisfies the assumptions \eqref{eq:2201} for $k=1,2$. 

Let 
$u_{0,k}^\vae \in L^2(\Omega)$ satisfy 
convergence \eqref{eq:3102} 
with 
$u_{0,k} \in L^2(\Omega)$. 
We assume that 
$u_k^\vae$ satisfies the problem:
\begin{equation}
\label{eq:332102}
\left\{
\begin{aligned}
&
\pp_t^\alpha u_k^\vae(x,t)
-
\dd (A_k^\vae(x)\nabla u_k^\vae(x,t))
=
0,
&(x,t)\in\Omega\times(0,T)
,\\&
u_k^\vae(x,t)=0,
&(x,t)\in\pp\Omega\times(0,T)
,\\&
u_k^\vae(x,0)=u_{0,k}^\vae(x),
&x\in\Omega
.
\end{aligned}
\right.
\end{equation}

By Theorem \ref{thm:3101} and Lemma \ref{lem:3103}, 
we may see that 
$u_k^\vae$ converges to
the weak solution $u_k^0$ of the following problem in the sense of the convergences \eqref{eq:3103} and \eqref{eq:3104}:
\begin{equation}
\label{eq:332103}
\left\{
\begin{aligned}
&
\pp_t^\alpha u_k^0(x,t)
-
\dd (A_k^0\nabla u_k^0(x,t))
=
0,
&(x,t)\in\Omega\times(0,T)
,\\&
u_k^0(x,t)=0,
&(x,t)\in\pp\Omega\times(0,T)
,\\&
u_k^0(x,0)=u_{0,k}(x),
&x\in\Omega
,
\end{aligned}
\right.
\end{equation}
where 
$A_k^0$ is the homogenized coefficient matrix defined by the following diagonal matrix:
\begin{equation*}
A_k^0
=
\begin{pmatrix}
p_k^0&0		&\cdots	&0	\\
0		&a_2^0	&		&\vdots\\
\vdots	&		&\ddots	&0\\
0		&\cdots	&0		&a_N^0
\end{pmatrix}
\end{equation*}
with
\begin{equation*}
p_k^0=\frac{1}{\M_{(0,\ell_1)}\left(\frac1{p_k}\right)}
,\quad 
a_i^0=\M_{(0,\ell_1)}\left(a_i\right)
,\quad
i=2,\ldots,N
\end{equation*}
for $k=1,2$.

Let us investigate the following inverse problem of type I. 

\noindent
\textbf{Inverse Problem I}:
Let us choose 
a sub-domain $\omega\subset\Omega$ and 
an open interval $I\subset (0,T)$. 
Determine the constant diffusion coefficient $p_k^0$ of the problem \eqref{eq:332103} from the data
of the problem \eqref{eq:332102}
\begin{equation*}
\int_I\int_\omega u_k^\vae(x,t)\,dxdt.
\end{equation*}

Using Theorem \ref{thm:3101} and Theorem \ref{thm:32204}, 
we may prove the following theorem 
by the same manner used in the proof of Theorem \ref{thm:33100}. 
In the following theorem, 
we assume that 
\eqref{eq:332103} is derived from 
\eqref{eq:332102} by the homogenization.

\begin{thm}[Asymptotic Stability]
\label{thm:33200}
Let 
$u_0:=u_{0,1}=u_{0,2}$ in $\Omega$. 
We assume that $u_0\in H^2(\Omega) \cap H^1_0(\Omega)$ satisfies 
\begin{equation*}
u_0\not\equiv 0 \quad \text{\rm in $\Omega$},\quad 
\pp_1^2 u_0(x) \ge 0\;\text{\rm and}\; 
\Delta u_0(x) \ge 0 \quad \text{\rm a.e. $x\in \Omega$}.
\end{equation*}
Then there exists a constant
$C=C(\nu,\mu) > 0$ such that 
\begin{equation*}
\left|
p_1^0-p_2^0
\right|
\leq 
C \left| 
\int_I\int_\omega u_1^\vae(x,t)\,dxdt
-
\int_I\int_\omega u_2^\vae(x,t)\,dxdt
\right|
+
\theta(\vae)
\end{equation*}
for all $\vae>0$. 
Here
$\theta(\vae)\to 0$ as $\vae\to 0$.
\end{thm}

Let us formulate our 
inverse problem of type II. 

\noindent
\textbf{Inverse Problem II-1}:
Let $x_0\in\Omega$ and $t_0\in (0,T)$ be arbitrarily chosen.
Determine the diffusion coefficient $p_k$ of the problem \eqref{eq:332102} from the data $u_k^0(x_0,t_0)$ of the problem \eqref{eq:332103}.

Calculating in the same way 
as Lemma \ref{lem:33101}, 
we obtain the following lemma 
on the relation between $p_k$ and $p_k^0$.
To prove the lemma, 
We assume the condition:
\begin{equation}
\label{eq:332101}
p_1(y_1)\geq p_2(y_1)
,\quad
y_1\in [0,\ell_1]
.
\end{equation}

In the followings, 
we assume that 
\eqref{eq:332103} is derived from 
\eqref{eq:332102} by the homogenization 
under its assumptions and \eqref{eq:332101}.

\begin{lem}
\label{lem:332101}
We have
\begin{equation*}
C^{-1}
\left\|
p_1-p_2
\right\|_{L^1(0,\ell_1)}
\leq
p_1^0-p_2^0
\leq
C
\left\|
p_1-p_2
\right\|_{L^1(0,\ell_1)}
\end{equation*}
with the constant $C>0$ depending on only $\nu,\mu,\ell_1$. 
\end{lem}

We may obtain the following stability result on our inverse problem 
by Theorem \ref{thm:32201} and Lemma \ref{lem:332101}. 

\begin{prop}[Stability]
Let 
$u_0:=u_{0,1}=u_{0,2}$ in $\Omega$. 
We assume that $u_0\in H^2(\Omega) \cap H^1_0(\Omega)$ satisfies 
\begin{equation*}
u_0\not\equiv 0 \quad \text{\rm in $\Omega$},\quad 
\pp_1^2 u_0(x) \ge 0\;\text{\rm and}\; 
\Delta u_0(x) \ge 0 \quad \text{\rm a.e. $x\in \Omega$}.
\end{equation*}
Then there exists a constant
$C=C(\nu,\mu, \ell_1) > 0$ such that 
\begin{equation*}
\|p_1-p_2\|_{L^1(0,\ell_1)} \leq C | u_1^0(x_0,t_0) - u_2^0(x_0,t_0)|.
\end{equation*}
\end{prop}

Now we investigate 
another inverse problem of type II. 

\noindent
\textbf{Inverse Problem II-2}:
Let $x_0\in\Omega$ and $t_1\in (0,T)$ be arbitrarily chosen. 
Determine the diffusion coefficient $p_k$ of the problem \eqref{eq:332102} from the data $u_k^0(x_0,t)$, $0<t<t_1$ of the problem \eqref{eq:332103}.

Then we have the uniqueness result with unknown initial values 
by Theorem \ref{thm:32203} and Lemma \ref{lem:332101}. 

\begin{prop}[Uniqueness with unknown initial values]
Let
\begin{equation*}
\mathcal{M}_{(0,\ell_1)}\left(\frac1{p_k}\right)\leq 1,
\end{equation*}
that is, $p_k^0\geq 1$ for $k=1,2$. 
We assume that $u_{0,k} \in H^2(\Omega) \cap H^1_0(\Omega)$ satisfy
\begin{equation*}
P_1 u_{0,1}(x_0)\ne 0\quad\text{\rm and}\quad
P_1 u_{0,2}(x_0)\ne 0
,
\end{equation*}
where 
$P_1$ is the eigenprojection defined in \S3.2.2.
If $u_{1}^0(x_0,t) = u_{2}^0(x_0,t)$ 
for $0<t<t_1$,
then we have $p_1=p_2$ a.e.\ on $(0,\ell_1)$.
\end{prop}

%
%
\section{Proofs of Main Results}
%
%
\subsection{Proof of Theorem \ref{thm:3101}}
We prove the result on the homogenization in \S3.1, Theorem \ref{thm:3101} by the oscillating test function method by Tartar \cite{T1978}. 

First 
we will consider the convergences \textrm{(i), (ii)} of \eqref{eq:3102} by using the boundedness of $\{u^\vae\}$. 
We note that 
$u^\vae$ satisfies the assumptions of 
Theorem \ref{thm:well-posedness_FDE} by 
$A^\vae\in M_S(\nu,\mu,\Omega)$ for each $\vae>0$. 
Hence the estimate \eqref{eq:estimate01} holds 
for $u^\vae \in W^\alpha(u_0^\vae)$, 
that is, 
\begin{equation*}
\|u^\vae\|_{\W^\alpha(u_0^\vae)}
\leq
C
\left(
\|u_0^\vae\|_{L^2(\Omega)}
+
\|f^\vae\|_{L^2(0,T;H^{-1}(\Omega))}
\right)
.
\end{equation*}
By the convergences \eqref{eq:3102} 
of $\{u_0^\vae\}$ and $\{f^\vae\}$, 
$\{u_0^\vae\}$ and $\{f^\vae\}$ are bounded in 
$L^2(\Omega)$ and $L^2(0,T;H^{-1}(\Omega))$,respectively. 
Moreover, by the equivalence of the norms \eqref{eq:2101}, 
we have
\begin{equation}
\label{eq:homo_pr01}
\|u^\vae\|_{L^2(0,T;H_0^1(\Omega))}
+
\|\pp_t^\alpha(u^\vae-u_0^\vae)\|_{L^2(0,T;H^{-1}(\Omega))}
\leq M
\end{equation}
with a positive constant $M$ 
which is independent of $\vae>0$.
Therefore, we may choose the subsequence of 
$\{u^\vae\}$  satisfying:
there exists $u^0\in L^2(0,T;H_0^1(\Omega))$
and
$\widehat{u}\in L^2(0,T;H^{-1}(\Omega))$ such that 
\begin{align}
\label{eq:homo_pr02}
&
u^\vae
\wto
u^0
\ \utext{weakly in $L^2(0,T;H_0^1(\Omega))$},
\\
\label{eq:homo_pr03}
&
\pp_t^\alpha(u^\vae-u_0^\vae)
\wto
\widehat{u}
\ \utext{weakly in $L^2(0,T;H^{-1}(\Omega))$}
\end{align}
as $\vae \to 0$.
Here, we still denote the subsequence by $\vae$. 
This is because we will show that 
the convergences hold for the whole sequences at the end of the proof.
Now we show that $\alpha$-th order fractional differentiability of $u^0-u_0$ 
with respect to $t$ and that 
$\widehat{u}=\pp_t^\alpha (u^0-u_0)$ 
by the argument used in Kubica and Yamamoto \cite{KY2018}. 
Using the weak convergence \eqref{eq:homo_pr03}, 
the property of the fractional derivative \eqref{eq:2102} and 
integrating by parts, 
for any $\psi\in \D(0,T)$ and 
$\Phi\in H_0^1(\Omega)$, we have
\begin{align*}
&
\int_0^T
\left\<
\widehat{u}(\cdot,t),\Phi
\right\>_{H^{-1}(\Omega),H_0^1(\Omega)}
\psi(t)
\,dt
\\&=
\lim_{\vae\to 0}
\int_0^T
\left\<
\pp_t^\alpha(u^\vae-u_0^\vae)(\cdot,t),\Phi
\right\>_{H^{-1}(\Omega),H_0^1(\Omega)}
\psi(t)
\,dt
\\&=
\lim_{\vae\to 0}
\int_0^T
\left\<
\pp_t J^{1-\alpha}(u^\vae-u_0^\vae)(\cdot,t),\Phi
\right\>_{H^{-1}(\Omega),H_0^1(\Omega)}
\psi(t)
\,dt
\\&=
\lim_{\vae\to 0}
\int_0^T
\pp_t
\left\<
J^{1-\alpha}(u^\vae-u_0^\vae)(\cdot,t),\Phi
\right\>_{H^{-1}(\Omega),H_0^1(\Omega)}
\psi(t)
\,dt
\\&=
-
\lim_{\vae\to 0}
\int_0^T
\int_\Omega
J^{1-\alpha}(u^\vae-u_0^\vae)(x,t)\Phi(x)
\psi^\prime(t)
\,dx
dt
\\&=
-
\int_0^T
\int_\Omega
J^{1-\alpha}(u^0-u_0)(x,t)\Phi(x)
\psi^\prime(t)
\,dx
dt
.
\end{align*}
In the last inequality, 
we used the weak continuity of $J^{1-\alpha}$ 
and the weak convergences \eqref{eq:homo_pr02} and \eqref{eq:3102} 
of $\{u^\vae\}$ and $\{u_0^\vae\}$. 
Thus we obtain
\begin{align*}
&
\int_0^T
\left\<
\widehat{u}(\cdot,t),\Phi
\right\>_{H^{-1}(\Omega),H_0^1(\Omega)}
\psi(t)
\,dt
\\&=
-
\int_0^T
\left\<
J^{1-\alpha}(u^0-u_0)(\cdot,t),\Phi
\right\>_{H^{-1}(\Omega),H_0^1(\Omega)}
\psi^\prime(t)
\,dt
.
\end{align*}
This implies that 
the fractional differentiability of $u^0-u_0$ 
in the weak sense and that
$\widehat{u}
=\pp_t J^{1-\alpha}(u^0-u_0)
=\pp_t^\alpha(u^0-u_0)\in L^2(0,T;H^{-1}(\Omega))$. 
Therefore, we may replace the convergence \eqref{eq:homo_pr03} by
\begin{equation}
\label{eq:homo_pr04}
\pp_t^\alpha(u^\vae-u_0^\vae)
\wto
\pp_t^\alpha(u^0-u_0)
\ \utext{weakly in $L^2(0,T;H^{-1}(\Omega))$}
.
\end{equation}

Since $\{u^\vae\}$ and $\{u_0^\vae\}$ are bounded 
in $\W^\alpha(u_0^\vae)$ and $L^2(\Omega)$, respectively, 
$\{u^\vae\}$ is bounded in $\mathrm{W}^\alpha$ 
by Remark \ref{rmk:2102}. 
Using Lemma \ref{lem:2101}, 
we can choose the subsequence of $\{u^\vae\}$ again (still denoted by $\vae$) such that 
\begin{equation}
\label{eq:homo_pr05}
u^\vae
\to
u^0
\ \utext{strongly in $L^2(0,T;L^2(\Omega))$},
\end{equation}
where the limit coincides with $u^0$ by the uniqueness of the weak limit. 

Next we consider the convergence \text{(iii)} of \eqref{eq:3102} and we will prove that
$u^0$ is the weak solution 
of the problem \eqref{eq:3105}. 
We define 
\begin{equation*}
\zeta^\vae(x,t)
=
\left(
\zeta_1^\vae(x,t)
,\ldots,
\zeta_N^\vae(x,t)
\right)
=
A^\vae(x)
\nabla u^\vae(x,t).
\end{equation*}
Since $A^\vae\in M_S(\nu,\mu,\Omega)$, we have
\begin{equation*}
\|\zeta^\vae\|_{(L^2(0,T;L^2(\Omega)))^N}
\leq \mu M
\end{equation*}
by the boundedness \eqref{eq:homo_pr01}. 
Hence we may choose the subsequence of $\{\zeta^\vae\}$ 
(still denoted by $\vae$) satisfying:
there exists $\zeta^0\in (L^2(0,T;L^2(\Omega)))^N$ such that 
\begin{equation}
\label{eq:homo_pr06}
\zeta^\vae
\wto
\zeta^0
\ \utext{weakly in $(L^2(0,T;L^2(\Omega)))^N$}
.
\end{equation}
By the definition of $\zeta^\vae$ and 
the variational formulation of the problem \eqref{eq:3101}, 
$\zeta^\vae$ satisfies
\begin{align*}
&
\int_0^T
\left\<
\pp_t^\alpha(u^\vae-u_0^\vae)(\cdot,t),\Phi
\right\>_{H^{-1}(\Omega),H_0^1(\Omega)}
\psi(t)
\,dt
+
\int_0^T
\int_\Omega
\zeta^\vae(x,t)
\cdot \nabla\Phi(x)
\psi(t)
\,dxdt
\\&=
\int_0^T
\left\<
f^\vae(\cdot,t),\Phi
\right\>_{H^{-1}(\Omega),H_0^1(\Omega)}
\psi(t)
\,dt
\quad\text{for any $\psi\in \D(0,T)$ and 
$\Phi\in H_0^1(\Omega)$},
\end{align*} 
that is, 
\begin{align}
\label{eq:homo_pr07}
&
\left\<
\pp_t^\alpha(u^\vae-u_0^\vae),\Phi\psi
\right\>_{L^2(0,T;H^{-1}(\Omega)),L^2(0,T;H_0^1(\Omega))}
\\\nonumber
&
+
\int_0^T
\int_\Omega
\zeta^\vae(x,t)
\cdot \nabla\Phi(x)
\psi(t)
\,dxdt
\\\nonumber
&=
\left\<
f^\vae,\Phi\psi
\right\>_{L^2(0,T;H^{-1}(\Omega)),L^2(0,T;H_0^1(\Omega))}
\ \text{for any $\psi\in \D(0,T)$ and 
$\Phi\in H_0^1(\Omega)$}
.
\end{align}
Taking the limit on the above equality 
by the weak convergences \textrm{(ii)} of \eqref{eq:3102}, 
\eqref{eq:homo_pr04} and \eqref{eq:homo_pr06}, 
we obtain
\begin{align}
\label{eq:homo_pr08}
&
\left\<
\pp_t^\alpha(u^0-u_0),\Phi\psi
\right\>_{L^2(0,T;H^{-1}(\Omega)),L^2(0,T;H_0^1(\Omega))}
\\\nonumber
&
+
\int_0^T
\int_\Omega
\zeta^0(x,t)
\cdot \nabla\Phi(x)
\psi(t)
\,dxdt
\\\nonumber
&=
\left\<
f,\Phi\psi
\right\>_{L^2(0,T;H^{-1}(\Omega)),L^2(0,T;H_0^1(\Omega))}
\ \text{for any $\psi\in \D(0,T)$ and 
$\Phi\in H_0^1(\Omega)$}.
\end{align}
Thus, it is sufficient to show that 
\begin{equation}
\label{eq:homo_pr09}
\zeta^0=A^0\nabla u^0
\end{equation}
to prove 
the convergence \text{(iii)} of \eqref{eq:3102} and 
that $u^0$ is the the weak solution 
of the problem \eqref{eq:3105}. 
Indeed, 
by substituting this equality \eqref{eq:homo_pr09} into the equation \eqref{eq:homo_pr08}, 
we may see that $u^0$ is the weak solution of the the problem \eqref{eq:3105}.

Let $\varphi\in\D(\Omega)$ and $\psi\in\D(0,T)$. 
Integrating the equality \eqref{eq:2211} on 
$(0,T)$ with $\Phi=u^\vae \varphi\psi$, 
we obtain
\begin{equation*}
\int_0^T
\int_\Omega
\eta_\xi^\vae(x)\cdot
\nabla u^\vae(x,t)\varphi(x)\psi(t)
\,dxdt
+
\int_0^T
\int_\Omega
\eta_\xi^\vae(x)\cdot
\nabla \varphi(x) u^\vae(x,t)\psi(t)
\,dxdt
=0
.
\end{equation*}
Hence we have
\begin{align*}
&
\int_0^T
\int_\Omega
\zeta^\vae(x,t)\cdot
\nabla w_\xi^\vae(x,t)\varphi(x)\psi(t)
\,dxdt
\\&=
\int_0^T
\int_\Omega
A^\vae(x)\nabla u^\vae(x,t)
\cdot
\nabla w_\xi^\vae(x)\varphi(x)\psi(t)
\,dxdt
\\&=
\int_0^T
\int_\Omega
A^\vae(x)\nabla w_\xi^\vae(x)
\cdot
\nabla u^\vae(x,t)\varphi(x)\psi(t)
\,dxdt
\\&=
\int_0^T
\int_\Omega
\eta_\xi^\vae(x)\cdot
\nabla u^\vae(x,t)\varphi(x)\psi(t)
\,dxdt
\\&=
-
\int_0^T
\int_\Omega
\eta_\xi^\vae(x)\cdot
\nabla \varphi(x) u^\vae(x,t)\psi(t)
\,dxdt
.
\end{align*}
Choosing $\Phi=\varphi w_\xi^\vae$ in the equality \eqref{eq:homo_pr07} 
and using the above equality, we obtain
\begin{align}
\label{eq:homo_pr10}
&
\left\<
\pp_t^\alpha(u^\vae-u_0^\vae), w_\xi^\vae\varphi\psi
\right\>_{L^2(0,T;H^{-1}(\Omega)),L^2(0,T;H_0^1(\Omega))}
\\\nonumber
&
+
\int_0^T
\int_\Omega
\zeta^\vae(x,t)
\cdot \nabla\varphi(x)
w_\xi^\vae\psi(t)
\,dxdt
\\\nonumber
&
-
\int_0^T
\int_\Omega
\eta_\xi^\vae(x)\cdot
\nabla \varphi(x) u^\vae(x,t)\psi(t)
\,dxdt
\\\nonumber
&=
\left\<
f^\vae,w_\xi^\vae\varphi\psi
\right\>_{L^2(0,T;H^{-1}(\Omega)),L^2(0,T;H_0^1(\Omega))}
.
\end{align}
To pass the limit in the above equation \eqref{eq:homo_pr10}, we show that 
\begin{align}
\label{eq:homo_pr11}
&
\lim_{\vae\to 0}
\left\<
\pp_t^\alpha(u^\vae-u_0^\vae), w_\xi^\vae\varphi\psi
\right\>_{L^2(0,T;H^{-1}(\Omega)),L^2(0,T;H_0^1(\Omega))}
\\\nonumber&=
\left\<
\pp_t^\alpha(u^0-u_0), (\xi\cdot x)\varphi\psi
\right\>_{L^2(0,T;H^{-1}(\Omega)),L^2(0,T;H_0^1(\Omega))}
.
\end{align}
By $\pp_t^\alpha(u^\vae-u_0^\vae)=\pp_t J^{1-\alpha}(u^\vae-u_0^\vae)$ and the integration by parts, we have
\begin{align*}
&
\lim_{\vae\to 0}
\left\<
\pp_t^\alpha(u^\vae-u_0^\vae), w_\xi^\vae\varphi\psi
\right\>_{L^2(0,T;H^{-1}(\Omega)),L^2(0,T;H_0^1(\Omega))}
\\&=
\lim_{\vae\to 0}
\int_0^T
\left\<
\pp_t J^{1-\alpha}(u^\vae-u_0^\vae)(\cdot,t), w_\xi^\vae\varphi
\right\>_{H^{-1}(\Omega),H_0^1(\Omega)}
\psi(t)
\,dt
\\&=
\lim_{\vae\to 0}
\int_0^T
\pp_t
\left\<
J^{1-\alpha}(u^\vae-u_0^\vae)(\cdot,t), w_\xi^\vae\varphi
\right\>_{H^{-1}(\Omega),H_0^1(\Omega)}
\psi(t)
\,dt
\\&=
-
\lim_{\vae\to 0}
\int_0^T
\left\<
J^{1-\alpha}(u^\vae-u_0^\vae)(\cdot,t), w_\xi^\vae\varphi
\right\>_{H^{-1}(\Omega),H_0^1(\Omega)}
\psi^\prime(t)
\,dt
\\&=
-
\lim_{\vae\to 0}
\int_0^T
\int_\Omega
J^{1-\alpha}(u^\vae-u_0^\vae)(x,t) w_\xi^\vae(x)\varphi(x)
\psi^\prime(t)
\,dxdt
\\&=
-
\int_0^T
\int_\Omega
J^{1-\alpha}(u^0-u_0)(x,t) 
(\xi\cdot x)\varphi(x)
\psi^\prime(t)
\,dxdt
\\&=
-
\int_0^T
\left\<
J^{1-\alpha}(u^0-u_0)(\cdot,t), 
(\xi\cdot x)\varphi
\right\>_{H^{-1}(\Omega),H_0^1(\Omega)}
\psi^\prime(t)
\,dt
\\&=
\int_0^T
\pp_t
\left\<
J^{1-\alpha}(u^0-u_0)(\cdot,t), 
(\xi\cdot x)\varphi
\right\>_{H^{-1}(\Omega),H_0^1(\Omega)}
\psi(t)
\,dt
\\&=
\int_0^T
\left\<
\pp_t
J^{1-\alpha}(u^0-u_0)(\cdot,t), 
(\xi\cdot x)\varphi
\right\>_{H^{-1}(\Omega),H_0^1(\Omega)}
\psi(t)
\,dt
\\&=
\int_0^T
\left\<
\pp_t^\alpha(u^0-u_0)(\cdot,t), 
(\xi\cdot x)\varphi
\right\>_{H^{-1}(\Omega),H_0^1(\Omega)}
\psi(t)
\,dt
\\&=
\left\<
\pp_t^\alpha(u^0-u_0), (\xi\cdot x)\varphi\psi
\right\>_{L^2(0,T;H^{-1}(\Omega)),L^2(0,T;H_0^1(\Omega))},
\end{align*}
where we used the weak continuity of $J^{1-\alpha}$ and 
convergences 
\eqref{eq:2209}, \eqref{eq:3102} and
\eqref{eq:homo_pr02}. Thus we may get \eqref{eq:homo_pr11}. 
Since all the terms in \eqref{eq:homo_pr10} 
except the one containing the time derivative 
are products of weakly and strongly convergent sequences, 
their limits are the products of the weakly and strongly convergent limits.
Hence we may pass the limit in \eqref{eq:homo_pr10}
by convergences 
\eqref{eq:2209}, \eqref{eq:2210}, 
\eqref{eq:3102}, 
\eqref{eq:homo_pr05},
\eqref{eq:homo_pr06} and
\eqref{eq:homo_pr11}. 
\begin{align*}
&
\left\<
\pp_t^\alpha(u^0-u_0), (\xi\cdot x)\varphi\psi
\right\>_{L^2(0,T;H^{-1}(\Omega)),L^2(0,T;H_0^1(\Omega))}
\\\nonumber
&
+
\int_0^T
\int_\Omega
\zeta^0(x,t)
\cdot \nabla\varphi(x)
(\xi\cdot x)\psi(t)
\,dxdt
\\\nonumber
&
-
\int_0^T
\int_\Omega
A^0\xi \cdot
\nabla \varphi(x) u^0(x,t)\psi(t)
\,dxdt
\\\nonumber
&=
\left\<
f,(\xi\cdot x)\varphi\psi
\right\>_{L^2(0,T;H^{-1}(\Omega)),L^2(0,T;H_0^1(\Omega))}
.
\end{align*}
This can be rewritten as follows.
\begin{align*}
&
\left\<
\pp_t^\alpha(u^0-u_0), (\xi\cdot x)\varphi\psi
\right\>_{L^2(0,T;H^{-1}(\Omega)),L^2(0,T;H_0^1(\Omega))}
\\\nonumber
&
+
\int_0^T
\int_\Omega
\zeta^0(x,t)
\cdot \nabla\left[(\xi\cdot x)\varphi(x)\right]
\psi(t)
\,dxdt
\\\nonumber
&
-
\int_0^T
\int_\Omega
\zeta^0(x,t)
\cdot \xi \varphi(x)
\psi(t)
\,dxdt
\\\nonumber
&
-
\int_0^T
\int_\Omega
A^0\xi\cdot
\nabla \varphi(x) u^0(x,t)\psi(t)
\,dxdt
\\\nonumber
&=
\left\<
f,(\xi\cdot x)\varphi\psi
\right\>_{L^2(0,T;H^{-1}(\Omega)),L^2(0,T;H_0^1(\Omega))}
.
\end{align*}
This together with 
the equality \eqref{eq:homo_pr08} with $\Phi=(\xi\cdot x)\varphi$ gives us
\begin{equation*}
\int_0^T
\int_\Omega
\zeta^0(x,t)
\cdot\xi\varphi(x)
\psi(t)
\,dxdt
=
-
\int_0^T
\int_\Omega
A^0\xi\cdot
\nabla \varphi(x) u^0(x,t)\psi(t)
\,dxdt
.
\end{equation*}
Noting that $A^0\xi$ is a constant, we have
\begin{equation*}
-
\int_0^T
\int_\Omega
A^0\xi\cdot
\nabla \varphi(x) u^0(x,t)\psi(t)
\,dxdt
=
\int_0^T
\int_\Omega
A^0\xi\cdot \nabla u^0(x,t)
\varphi(x) \psi(t)
\,dxdt
.
\end{equation*}
By the above two equalities, we obtain
\begin{equation*}
\int_0^T
\int_\Omega
\zeta^0(x,t)
\cdot\xi\varphi(x)
\psi(t)
\,dxdt
=
\int_0^T
\int_\Omega
A^0\xi\cdot \nabla u^0(x,t)
\varphi(x) \psi(t)
\,dxdt
.
\end{equation*}
Therefore, we see that
\begin{equation*}
\zeta^0\cdot\xi
=
A^0\xi\cdot \nabla u^0
=
A^0\nabla u^0\cdot\xi.
\end{equation*}
Since $\xi\in\mathbb{R}^N$ is arbitrary, we have \eqref{eq:homo_pr09}. 

At the end of the proof, 
we observe that all convergences hold for the whole sequence $\{u^\vae\}$. 
By \eqref{eq:homo_pr08} and \eqref{eq:homo_pr09}, 
we have the weak formulation of the problem \eqref{eq:3105}. 
We also see that the weak solution $u^0$ of the problem \eqref{eq:3105} is unique by Remark \ref{rmk:3101}. 
Therefore, convergences 
\eqref{eq:homo_pr02}, 
\eqref{eq:homo_pr04}, 
\eqref{eq:homo_pr05} and  
\eqref{eq:homo_pr06} hold for the whole sequence $\{u^\vae\}$, that is, 
$u^\vae$ satisfies \eqref{eq:3103} and \eqref{eq:3104}. 

Thus we complete the proof of Theorem \ref{thm:3101}. \qed

%
%
\subsection{Proofs of Theorems on Determination of Constant Diffusion Coefficient}

We show the results on the determinations 
of the constant diffusion coefficient for time-fractional diffusion equations: 
Theorems \ref{thm:32101}, \ref{thm:32102}, \ref{thm:32103}, \ref{thm:32104}, and
\ref{thm:32201}, \ref{thm:32203}, \ref{thm:32204} in \S3.2.

\subsubsection{Proof of Theorem \ref{thm:32101}}
We will prove this theorem in three steps. 

\noindent
\textbf{First Step.}
We show the following inequality: if $p\geq q$ then we have
\begin{equation}
\label{eq:4201}
u_p(x,t)\geq u_q (x,t),\quad
(x,t)\in\Omega\times (0,\infty).
\end{equation}
Set $y:= u_p - u_q$ and $r:= p-q$.  Then 
\begin{equation}
\label{eq:4202}
\left\{
\begin{aligned}
&
\pp_t^\alpha y(x,t)
-
p \Delta y(x,t)
=
r \Delta u_q(x,t),
&(x,t)\in\Omega\times(0,T)
,\\&
y(x,t)=0,
&(x,t)\in\pp\Omega\times(0,T)
,\\&
y(x,0)=0,
&x\in\Omega
.
\end{aligned}
\right.
\end{equation}
Setting $v:= \Delta u_q$, by $u_0 \in H^2(\Omega) \cap H^1_0(\Omega)$,
we may observe
\begin{equation}
\label{eq:4203}
\left\{
\begin{aligned}
&
\pp_t^\alpha v(x,t)
-
q \Delta v(x,t)
=
0,
&(x,t)\in\Omega\times(0,T)
,\\&
v(x,t)=0,
&(x,t)\in\pp\Omega\times(0,T)
,\\&
v(x,0)=\Delta u_0(x),
&x\in\Omega
.
\end{aligned}
\right.
\end{equation}
Indeed, since
\begin{equation*}
u_q(x,t) = \sum_{n=1}^\infty E_{\alpha,1}(-q\la_n t^{\alpha})P_n u_0(x)
,\quad
(x,t)\in\Omega\times (0,\infty),
\end{equation*}
we have
\begin{align*}
- \Delta u_q(x,t) 
&=
\sum_{n=1}^\infty E_{\alpha,1}(-q\la_n t^{\alpha})(\mathcal{A}P_n u_0)(x)
\\&= 
-\sum_{n=1}^\infty E_{\alpha,1}(-q\la_nt^{\alpha})P_n(\Delta u_0)(x)
,\quad
(x,t)\in\Omega\times (0,\infty).
\end{align*}
Thus we may verify \eqref{eq:4203}. 

Applying Theorem 2.1 in Luchko and Yamamoto \cite{LY2017} with 
\eqref{eq:32102} to \eqref{eq:4203}, we see that
\begin{equation}
\label{eq:4204}
\Delta u_q \geq 0\quad\text{\rm in $\Omega\times (0,\infty)$}.
\end{equation}
Again we use Theorem 2.1 in \cite{LY2017} to \eqref{eq:4202}. 
Noting that the right-hand side of the equation in \eqref{eq:4202} is non-negative: $r\Delta u_q \geq 0$ in $\Omega\times (0,\infty)$, 
we obtain that $y\geq 0$ in $\Omega\times (0,\infty)$, 
which means the inequality \eqref{eq:4201}

\noindent
\textbf{Second Step.}
We confirm the analyticity of $u_p(x_0,t_0)$ with respect to $p>0$. 
Set $\mathcal{A}_{p}:= -p\Delta$ with 
$\D(\mathcal{A}_{p}) = H^2(\Omega) \cap H^1_0(\Omega)$.
The eigenvalues are $p\la_n$ and the eigenprojection is the same as 
$\mathcal{A}_1$.
Then we have
\begin{equation}
\label{eq:4205}
u_p(x_0,t_0) = \sum_{n=1}^\infty E_{\alpha,1}(-p\la_n t_0^{\alpha})P_n u_0(x_0)
\end{equation}
and the series is convergent.
Indeed 
\begin{equation*}
\mathcal{A}_{p} u_p(x,t) 
=
\sum_{n=1}^\infty p\la_n E_{\alpha,1}(-p\la_n t^{\alpha}) P_n u_0(x),
\end{equation*}
and so 
\begin{align*}
\| \mathcal{A}_{p} u_p(\cdot,t)\|^2_{L^2(\Omega)}
&\leq
C\sum_{n=1}^\infty \| P_n u_0\|^2_{L^2(\Omega)}\left( \frac{p\la_n t^{\alpha}}
{1+p\la_n t^{\alpha}}\right)^2t^{-2\alpha}
\\&\leq 
Ct^{-2\alpha} \sum_{n=1}^\infty \| P_n u_0\|^2_{L^2(\Omega)} 
\\&= Ct^{-2\alpha}\|u_0\|^2_{L^2(\Omega)}.
\end{align*}
Since $H^2(\Omega) \subset C(\overline{\Omega})$ 
by the Sobolev embedding and 
$N=1,2,3$, we obtain
\begin{equation*}
| u_p(x,t)| 
\leq
Ct^{-\alpha}\| u_0 \|_{L^2(\Omega)}
,\quad 
(x,t)\in\Omega\times (0,\infty). 
\end{equation*}
Thus we obtain the convergent series \eqref{eq:4205}. 
See also Theorem 2.1 in Sakamoto and Yamamoto \cite{SY2011}.

Since $E_{\alpha,1}(-(\la_n t_0^{\alpha})z)$ is holomorphic in 
$z\in \mathbb{C}$, we readily see that 
\begin{equation}
\label{eq:4206}
\mbox{ $u_p(x_0,t_0)$ is analytic in $p>0$.}
\end{equation}

As a preparation for the third step, 
we show the following inequality. 
\begin{equation}
\label{eq:4207}
u_0(x) \leq 0 \quad x\in \Omega.
\end{equation}
Let 
\begin{equation*}
u_0(\widehat{x}):=\max_{x\in \overline{\Omega}} u_0(x)>0.
\end{equation*}
Since $u_0=0$ on $\pp\Omega$, we see that 
$\widehat{x} \in \Omega$.
By $\Delta u_0 \ge 0$ in $\Omega$, the strong maximum principle 
(e.g., Theorem 3.5 (p.35) in Gilbarg and Trudinger \cite{GT}) 
implies that $u_0$ is a constant function.  By $\bigl. u_0\bigr|_{\pp\Omega} = 0$,
we obtain $u_0=0$ in $\Omega$.  This is a contradiction by the assumption 
\eqref{eq:32102}.  Thus \eqref{eq:4207} is proved.

\noindent
\textbf{Third Step.}
We prove the inequality: if $p> q$ then we have
\begin{equation}
\label{eq:4208}
u_p(x_0,t_0)> u_q(x_0,t_0).
\end{equation}
Assume that 
there exists constants
$p_0$ and $q_0$ satisfying $0<q_0<p_0$ and $u_{p_0}(x_0,t_0)=u_{q_0}(x_0,t_0)$. 
In view of \eqref{eq:4201}, we obtain 
$u_s(x_0,t_0)=u_{q_0}(x_0,t_0)$ for all $s\in [q_0,p_0]$. 
Using the the identity theorem, the analyticity \eqref{eq:4206} yields
\begin{equation*}
u_p(x_0,t_0)=u_q(x_0,t_0)
\quad
\text{\rm for all $p>q>0$}.
\end{equation*}
Hence we have
\begin{equation*}
\sum_{n=1}^\infty E_{\alpha,1}(-p\la_nt_0^{\alpha}) P_n u_0(x_0)
= 
\sum_{n=1}^\infty E_{\alpha,1}(-q\la_nt_0^{\alpha}) P_n u_0(x_0)
\quad
\text{\rm for all $p>q>0$}.
\end{equation*}
By the asymptotic expansion (e.g., \cite{Po}), for large $p>0$ we have
\begin{equation*}
\left|
\sum_{n=1}^\infty E_{\alpha,1}(-p\la_n t_0^{\alpha}) P_n u_0(x_0)
\right|
\le \frac{C}{pt_0^{\alpha}}.
\end{equation*}
Setting $q=1$, we have
\begin{equation*}
u_q(x_0,t_0) = \lim_{p\to \infty} u_p(x_0,t_0) = 0.
\end{equation*}
On the other hand, in view of \eqref{eq:4207}, the strong positivity (e.g., Theorem 9 in 
Luchko and Yamamoto \cite{LY2019}) yields that 
$u_q(x_0,t_0) < 0$.   This is a contradiction.
Thus the proof of \eqref{eq:4208} is complete.

From the inequality \eqref{eq:4208}, 
we may see that the function $h(p):=u_p(x_0,t_0)$ is injective and 
\begin{equation*}
\frac{dh}{dp}(p)>0
\quad \text{\rm for all $p>0$}. 
\end{equation*}
Hence the mean-value theorem yields the stability estimate. 
Thus the proof of Theorem \ref{thm:32101} is complete. \qed

\subsubsection{Proof of Theorem \ref{thm:32104}}

The first and second steps of the proof are 
the same as in the proof of Theorem \ref{thm:32101}. 
Therefore we begin the proof with the third step. 

\noindent
\textbf{Third Step.}
We will show the inequality: if $p> q$ then we have
\begin{equation}
\label{eq:42301}
\int_I\int_\omega
u_p(x,t)\,dxdt
>
\int_I\int_\omega
u_q(x,t)
\,dxdt.
\end{equation}
Assume that 
there exists constants
$p_0$ and $q_0$ satisfying $0<q_0<p_0$ and 
\begin{equation*}
\int_I\int_\omega
u_{p_0}(x,t)\,dxdt
=
\int_I\int_\omega
u_{q_0}(x,t)
\,dxdt.
\end{equation*}
By \eqref{eq:4201}, we have 
\begin{equation*}
u_{p_0}(x,t)\geq u_{q_0} (x,t),\quad
(x,t)\in\Omega\times (0,\infty).
\end{equation*}
Since $u_{p_0}$ and $u_{q_0}$ are smooth enough (see e.g., Theorem 2.1 in \cite{SY2011}), 
we obtain
\begin{equation*}
u_{p_0}(x,t)= u_{q_0} (x,t),\quad
(x,t)\in\omega\times I.
\end{equation*}
Let us fix $(x_0,t_0)\in\omega\times I$ arbitrarily. 
We have 
\begin{equation*}
u_{p_0}(x_0,t_0)= u_{q_0} (x_0,t_0), 
\end{equation*}
which implies 
\eqref{eq:42301} by 
an argument similar to that used 
in the third step of the proof of Theorem \ref{thm:32101}.

From the inequality \eqref{eq:42301}, 
the function
\begin{equation*}
H(p)=
\int_I\int_\omega u_p(x,t)\,dxdt
\end{equation*}
is injective by the inequality \eqref{eq:42301}, and
\begin{equation*}
\frac{dH}{dp}(p)>0,\quad p>0.
\end{equation*}
This yields the stability estimate by the same argument used in the proof of Theorem \ref{thm:32101}. 

Thus the proof of Theorem \ref{thm:32104} is complete. \hfill\qed

\subsubsection{Proofs of Theorems \ref{thm:32102} and \ref{thm:32103}}
Before we begin the proofs of the theorems, 
we confirm a few facts which are required to prove theorems.

The main ingredients are 
\begin{equation*}
\widetilde{u}_p(x,t)
=
\sum_{n=1}^\infty E_{\alpha,1}(-p\lambda_n t^\alpha) P_n u_{0,p}(x)
,\quad
(x,t)\in\Omega\times (0,\infty)
\end{equation*}
and the asymptotic expansion:
\begin{equation*}
E_{\alpha,1}(-p\lambda_n t^\alpha)
=
\sum_{k=1}^K
\frac{(-1)^{k+1}}{\Gamma(1-\alpha k)}
\frac{1}{p^k \lambda_n^k t^{\alpha k}}
+
O\left( \frac{1}{p^{K+1}\la_n^{K+1}t^{\alpha(K+1)}}\right)
\end{equation*}
for any $K\in \mathbb{N}$.

We assume
\begin{equation*}
\widetilde{u}_p(x_0,t)
=
\widetilde{u}_q(x_0,t)
,\quad
0<t<t_1.
\end{equation*}
The analyticity of the solution in $t>0$ yields
\begin{equation*}
\widetilde{u}_p(x_0,t)
=
\widetilde{u}_q(x_0,t)
,\quad
t>0.
\end{equation*}
Therefore
\begin{equation*}
\sum_{n=1}^\infty E_{\alpha,1}(-p\la_n t^{\alpha}) P_n u_{0,p}(x_0) 
= 
\sum_{n=1}^\infty E_{\alpha,1}(-q\la_n t^{\alpha}) P_n u_{0,q}(x_0),\quad t>0.
\end{equation*}
Set
\begin{equation}
\label{eq:42085}
\mathcal{P}_k (x) := \sum_{n=1}^\infty \frac{P_n u_{0,p}(x)}{\la_n^k}, \quad
\mathcal{Q}_k (x):= \sum_{n=1}^\infty \frac{P_n u_{0,q}(x)}{\la_n^k}
\end{equation}
for $x\in \Omega$.
Hence,
\begin{equation}
\label{eq:4209}
\sum_{k=1}^K \frac{(-1)^{k+1}}{\Gamma(1-\alpha k)}
\frac{1}{p^kt^{\alpha k}}\mathcal{P}_k (x_0)
= 
\sum_{k=1}^K \frac{(-1)^{k+1}}{\Gamma(1-\alpha k)}
\frac{1}{q^kt^{\alpha k}}\mathcal{Q}_k (x_0)
+ O\left( \frac{1}{t^{\alpha(K+1)}}\right)
\end{equation}
as $t \to \infty$.

On the other hand, we can show
\begin{lem}
\label{lem:4201}
Let 
$\displaystyle
\sum_{n=1}^\infty \left| \frac{r_n}{\la_n}\right| < \infty
$ and 
let a sequence 
$\{ \ell_k\}_{k\in \mathbb{N}}$ satisfy 
$\displaystyle
\lim_{k\to \infty} \ell_k = \infty
$.
If 
\begin{equation*}
\sum_{n=1}^\infty \frac{r_n}{\la_n^{\ell_k}} = 0 \quad \text{\rm for all $k\in \mathbb{N}$},
\end{equation*}
then $r_n=0$ for $n\in \mathbb{N}$.
\end{lem}
The proof is direct and see Lemma 1 in Yamamoto \cite{Y2021}.

Under the assumptions respectively in Theorems \ref{thm:32102} and \ref{thm:32103}, 
we verify
\begin{equation}
\label{eq:4210}
\left\{
\begin{aligned}
&\text{there exists $k_0\in \mathbb{N}$ such that }
\\&
\text{$\mathcal{P}_k(x_0)\ne 0$ for $k\ge k_0$ or $\mathcal{Q}_k(x_0)\ne 0$ for $k\ge k_0$.}
\end{aligned}
\right.
\end{equation}
Indeed, if not, then we can choose a sequence 
$\ell_k$ with 
$k\in \mathbb{N}$ such that 
$\displaystyle
\lim_{k\to\infty} \ell_k = \infty
$ and 
$\mathcal{P}_{\ell_k}(x_0) = \mathcal{Q}_{\ell_k}(x_0) = 0$ for $k\in \mathbb{N}$.  Application of Lemma \ref{lem:4201}
yields $P_n u_{0,p}(x_0) = P_n u_{0,q}(x_0) = 0$ for $n\in \mathbb{N}$, which is impossible 
by the assumptions in both Theorems \ref{thm:32102} and \ref{thm:32103}.

\begin{proof}[Proof of Theorem \ref{thm:32102}]
We set $u_{0,p}\equiv u_{0,q}$. 
By \eqref{eq:4210}, 
we can find $K_1 \in \mathbb{N}$ such that 
$\mathcal{P}_1 (x_0)= \cdots \mathcal{P}_{K_1-1}(x_0) = 0$ and $\mathcal{P}_{K_1}(x_0) \ne 0$.  
For $K_1=1$, we can simply have $\mathcal{P}_1\ne 0$.
Then by \eqref{eq:4209}, we obtain
\begin{equation*}
\frac{(-1)^{K_1+1}}{\Gamma(1-\alpha K_1)}
\frac{\mathcal{P}_{K_1}(x_0)}{p^{K_1}t^{\alpha K_1}}
= 
\frac{(-1)^{K_1+1}}{\Gamma(1-\alpha K_1)}
\frac{\mathcal{P}_{K_1}(x_0)}{q^{K_1}t^{\alpha K_1}}
+ O\left( \frac{1}{t^{\alpha(K_1+1)}}\right)
\end{equation*}
as $t \to \infty$. 
Letting $t\to \infty$, we reach 
\begin{equation*}
\mathcal{P}_{K_1}(x_0) \left( \frac{1}{p^{K_1}} - \frac{1}{q^{K_1}}
\right) = 0,
\end{equation*}
that is, $p=q$ follows by $\mathcal{P}_{K_1}(x_0) \ne 0$.
The proof of Theorem \ref{thm:32102} is complete.
\end{proof}
\begin{proof}[Proof of Theorem \ref{thm:32103}]
By \eqref{eq:4209}, 
equating the coefficients of $t^{-\alpha k}$, 
we have
\begin{equation*}
\frac{\mathcal{P}_k(x_0)}{p^k} = \frac{\mathcal{Q}_k(x_0)}{q^k}, \quad k\in \mathbb{N},
\end{equation*}
that is,
\begin{equation*}
\sum_{n=1}^\infty \frac{P_n u_{0,p}(x_0)}{\la_n^k} = \sum_{n=1}^\infty \frac{P_n u_{0,q}(x_0)}{(\la_n\rho)^k},
\quad k\in \mathbb{N},
\end{equation*}
where we set $\rho := \frac{q}{p}$.  We assume that $\rho=\frac{q}{p} > 1$.
Then
\begin{equation*}
\frac{P_1 u_{0,p}(x_0)}{\la_1^k}
+ 
\sum_{n=2}^{\infty} \frac{P_n u_{0,p}(x_0)}{\la_n^k} 
= 
\sum_{n=1}^\infty \frac{P_n u_{0,q}(x_0)}{(\la_n\rho)^k}.
\end{equation*}
Hence,
\begin{equation*}
P_1 u_{0,p}(x_0)
+ \sum_{n=2}^{\infty} \left( \frac{\la_1}{\la_n}\right)^k P_n u_{0,p}(x_0) 
= \sum_{n=1}^\infty \left( \frac{\la_1}{\la_n\rho}\right)^k P_n u_{0,q}(x_0).
\end{equation*}
Since $\left| \frac{\la_1}{\la_n}\right| < 1$ 
for $n\ge 2$ and
$\left| \frac{\la_1}{\rho\la_n}\right| < 1$ for 
$n\ge 1$, we see
that 
\begin{equation*}
\lim_{k\to \infty} \left( \frac{\la_1}{\la_n}\right)^k = 0 
\quad \text{\rm for $n\ge 2$}
\end{equation*}
and
\begin{equation*}
\lim_{k\to \infty} \left( \frac{\la_1}{\rho\la_n}\right)^k = 0 
\quad \text{\rm for $n\ge 1$}.
\end{equation*}
Therefore we reach $P_1 u_{0,p}(x_0) = 0$. 
This is impossible by the assumption. 
Consequently 
$\frac{q}{p} \le 1$.  
Similarly we can prove 
$
\frac{p}{q} \ge 1
$ 
and $p=q$ follows.
Thus the proof of Theorem \ref{thm:32103} is complete.
\end{proof}

\subsubsection{Proof of Theorem \ref{thm:32201}}

We prove this theorem using arguments similar to those in Theorem \ref{thm:32101}.

\noindent
\textbf{First Step.}
As we have seen in the first step of the proof of Theorem \ref{thm:32101}, 
we may see that 
\begin{equation}
\label{eq:42401}
u_p(x,t)\geq u_q (x,t),\quad
(x,t)\in\Omega\times (0,\infty)
,
\end{equation}
if $p\geq q$. 

Set $y:= u_p - u_q$ and $r:= p-q$.  Then 
\begin{equation}
\label{eq:42402}
\left\{
\begin{aligned}
&
\pp_t^\alpha y(x,t)
-
\dd(B_p\nabla y(x,t))
=
r \pp_1^2 u_q(x,t),
&(x,t)\in\Omega\times(0,T)
,\\&
y(x,t)=0,
&(x,t)\in\pp\Omega\times(0,T)
,\\&
y(x,0)=0,
&x\in\Omega
.
\end{aligned}
\right.
\end{equation}
Setting $v:= \pp_1^2 u_q$, by $u_0 \in H^2(\Omega) \cap H^1_0(\Omega)$,
we may see that
\begin{equation}
\label{eq:42403}
\left\{
\begin{aligned}
&
\pp_t^\alpha v(x,t)
-
\dd(B_q \nabla v(x,t))
=
0,
&(x,t)\in\Omega\times(0,T)
,\\&
v(x,t)=0,
&(x,t)\in\pp\Omega\times(0,T)
,\\&
v(x,0)=\pp_1^2 u_0(x),
&x\in\Omega
.
\end{aligned}
\right.
\end{equation}

Let us apply Theorem 2.1 in \cite{LY2017} with 
\eqref{eq:32205} to \eqref{eq:42403}, we see 
\begin{equation}
\label{eq:42404}
\pp_1^2 u_q \geq 0\quad\text{\rm in $\Omega\times (0,\infty)$}.
\end{equation}
Using Theorem 2.1 in \cite{LY2017} to \eqref{eq:42402}. 
By $r\pp_1^2 u_q \geq 0$ in $\Omega\times (0,\infty)$, 
we see that $y\geq 0$ in $\Omega\times (0,\infty)$, 
which implies the inequality \eqref{eq:42401}

\noindent
\textbf{Second Step.}
We observe the analyticity of $u_p(x_0,t_0)$ with respect to $p>0$. 
We have
\begin{equation}
\label{eq:42405}
u_p(x_0,t_0) 
=
\sum_{n=1}^\infty\sum_{j=1}^{d_n} 
E_{\alpha,1}((-p\Lambda_{n_j}-\widetilde\Lambda_{n_j}) t_0^{\alpha})
(u_0,\varphi_{n_j})_{L^2(\Omega)}\varphi_{n_j}(x_0)
\end{equation}
and the series is convergent.
Indeed 
\begin{equation*}
-\dd(B_p\nabla u_p(x,t)) 
=
\sum_{n=1}^\infty\sum_{j=1}^{d_n} 
(p\Lambda_{n_j}+\widetilde\Lambda_{n_j})
E_{\alpha,1}((-p\Lambda_{n_j}-\widetilde\Lambda_{n_j}) t^{\alpha})
(u_0,\varphi_{n_j})_{L^2(\Omega)}\varphi_{n_j}(x)
,
\end{equation*}
and so 
\begin{align*}
\left\| -\dd(B_p\nabla u_p(\cdot,t)) 
\right\|^2_{L^2(\Omega)}
&\leq
C
\sum_{n=1}^\infty\sum_{j=1}^{d_n} 
(u_0,\varphi_{n_j})_{L^2(\Omega)}^2
\left( 
\frac{(p\Lambda_{n_j}+\widetilde\Lambda_{n_j})t^\alpha}
{1+(p\Lambda_{n_j}+\widetilde\Lambda_{n_j})t^\alpha}\right)^2t^{-2\alpha}
\\&\leq 
Ct^{-2\alpha} \sum_{n=1}^\infty \| P_n u_0\|^2_{L^2(\Omega)} 
\\&= Ct^{-2\alpha}\|u_0\|^2_{L^2(\Omega)}.
\end{align*}
Since $H^2(\Omega) \subset C(\overline{\Omega})$ 
by the Sobolev embedding and 
$N=1,2,3$, we obtain
\begin{equation*}
| u_p(x,t)| 
\leq
Ct^{-\alpha}\| u_0 \|_{L^2(\Omega)}
,\quad 
(x,t)\in\Omega\times (0,\infty). 
\end{equation*}
Thus we obtain the convergent series \eqref{eq:42405}. 

Since $E_{\alpha,1}((-z\Lambda_{n_j}-\widetilde\Lambda_{n_j})t_0^\alpha)$ is holomorphic in 
$z\in \mathbb{C}$, we readily see that 
\begin{equation}
\label{eq:42406}
\mbox{ $u_p(x_0,t_0)$ is analytic in $p>0$.}
\end{equation}

As we have shown in the proof of Theorem \ref{thm:32201}, 
we have the following inequality. 
\begin{equation}
\label{eq:42407}
u_0(x) \leq 0 \quad x\in \Omega.
\end{equation}

\noindent
\textbf{Third Step.}
We prove the inequality: if $p> q$ then we have
\begin{equation}
\label{eq:42408}
u_p(x_0,t_0)> u_q(x_0,t_0).
\end{equation}
Assume that 
there exists constants
$p_0$ and $q_0$ satisfying $0<q_0<p_0$ and $u_{p_0}(x_0,t_0)=u_{q_0}(x_0,t_0)$. 
By the argument used in the proof of Theorem \ref{thm:32201},  we see that
\begin{equation*}
u_p(x_0,t_0)=u_q(x_0,t_0)
\quad
\text{\rm for all $p>q>0$}.
\end{equation*}
Hence we have
\begin{align*}
&
\sum_{n=1}^\infty\sum_{j=1}^{d_n} 
E_{\alpha,1}((-p\Lambda_{n_j}-\widetilde\Lambda_{n_j}) t_0^{\alpha})
(u_0,\varphi_{n_j})_{L^2(\Omega)}\varphi_{n_j}(x_0)
\\&= 
\sum_{n=1}^\infty\sum_{j=1}^{d_n} 
E_{\alpha,1}((-q\Lambda_{n_j}-\widetilde\Lambda_{n_j}) t_0^{\alpha})
(u_0,\varphi_{n_j})_{L^2(\Omega)}\varphi_{n_j}(x_0)
\quad
\text{\rm for all $p>q>0$}.
\end{align*}
By the asymptotic expansion (e.g., \cite{Po}), for large $p>0$ we have
\begin{equation*}
\left|
\sum_{n=1}^\infty\sum_{j=1}^{d_n} 
E_{\alpha,1}((-p\Lambda_{n_j}-\widetilde\Lambda_{n_j}) t_0^{\alpha})
(u_0,\varphi_{n_j})_{L^2(\Omega)}\varphi_{n_j}(x_0)
\right|
\le \frac{C}{(p\Lambda_{1}+\widetilde\Lambda_{1})t_0^{\alpha}}.
\end{equation*}
Here $p\Lambda_1+\widetilde\Lambda_1$ is the first eigenvalue for $p>1$ and this is simple. Indeed, we have
\begin{align}
\label{eq:42409}
p\Lambda_1+\widetilde\Lambda_1
&=
(p-1)\Lambda_1 + (\Lambda_1+\widetilde\Lambda_1)
\\\nonumber&<
(p-1)\Lambda_{n_j} + (\Lambda_{n_j}+\widetilde\Lambda_{n_j})
\\\nonumber&=
p\Lambda_{n_j}+\widetilde\Lambda_{n_j}, 
\end{align}
for $n \geq 2$, 
since we know 
\begin{equation*}
\Lambda_1 \leq \Lambda_{n_j}
\quad
\text{and}
\quad
\Lambda_1+\widetilde\Lambda_1
=\la_1
<\la_n
=
\Lambda_{n_j}+\widetilde\Lambda_{n_j}.
\end{equation*}
Setting $q=1$, we have
\begin{equation*}
u_q(x_0,t_0) = \lim_{p\to \infty} u_p(x_0.t_0) = 0.
\end{equation*}
On the other hand, we may obtain
$u_q(x_0,t_0) < 0$ by the same manner used in Third step in the proof of Theorem \ref{thm:32201}.   This is a contradiction.
Thus the proof of \eqref{eq:42408} is complete.

From the inequality \eqref{eq:42408}, 
we may see that the function $h(p):=u_p(x_0,t_0)$ is injective and 
\begin{equation*}
\frac{dh}{dp}(p)>0
\quad \text{\rm for all $p>0$}. 
\end{equation*}
Hence the mean-value theorem yields the stability estimate. 
Thus the proof of Theorem \ref{thm:32201} is complete. \qed

\subsubsection{Proof of Theorem \ref{thm:32204}}

The first and second steps of the proof are 
the same as the proof of Theorem \ref{thm:32201}. 
Moreover, we may obtain the stability estimate by 
the argument similar to that used in the third step of the proof of Theorem \ref{thm:32104}.

Thus the proof of Theorem \ref{thm:32204} is complete. \hfill\qed

\subsubsection{Proof of Theorem \ref{thm:32203}}
We use
\begin{equation*}
\widetilde{u}_p(x,t)
=
\sum_{n=1}^\infty 
\sum_{j=1}^{d_n}
E_{\alpha,1}((-p\Lambda_{n_j}-\widetilde\Lambda_{n_j} )t^\alpha) (u_{0,p},\varphi_{n,j})_{L^2(\Omega)}\varphi_{n_j}(x)
,\quad
(x,t)\in\Omega\times (0,\infty)
\end{equation*}
and the asymptotic expansion:
\begin{align*}
E_{\alpha,1}((-p\Lambda_{n_j}-\widetilde\Lambda_{n_j} )t^\alpha)
&=
\sum_{k=1}^K
\frac{(-1)^{k+1}}{\Gamma(1-\alpha k)}
\frac{1}{(p\Lambda_{n_j}+\widetilde\Lambda_{n_j})^k t^{\alpha k}}
\\&\quad+
O\left( \frac{1}{(p\Lambda_{n_j}+\widetilde\Lambda_{n_j})^{K+1}t^{\alpha(K+1)}}\right)
\end{align*}
for any $K\in \mathbb{N}$.

We assume
\begin{equation*}
\widetilde{u}_p(x_0,t)
=
\widetilde{u}_q(x_0,t)
,\quad
0<t<t_1.
\end{equation*}
By the analyticity of the solution in $t>0$, we have
\begin{equation*}
\widetilde{u}_p(x_0,t)
=
\widetilde{u}_q(x_0,t)
,\quad
t>0.
\end{equation*}
Therefore
\begin{align*}
&\sum_{n=1}^\infty 
\sum_{j=1}^{d_n}
E_{\alpha,1}((-p\Lambda_{n_j}-\widetilde\Lambda_{n_j} )t^\alpha) (u_{0,p},\varphi_{n,j})_{L^2(\Omega)}\varphi_{n_j}(x_0)
\\&= 
\sum_{n=1}^\infty 
\sum_{j=1}^{d_n}
E_{\alpha,1}((-q\Lambda_{n_j}-\widetilde\Lambda_{n_j} )t^\alpha) (u_{0,q},\varphi_{n,j})_{L^2(\Omega)}\varphi_{n_j}(x_0)
,\quad t>0.
\end{align*}
Hence,
\begin{align}
&\label{eq:42410}
\sum_{k=1}^K \frac{(-1)^{k+1}}{\Gamma(1-\alpha k)}
\frac{1}{t^{\alpha k}}
\sum_{n=1}^\infty 
\sum_{j=1}^{d_n}
\frac{(u_{0,p},\varphi_{n,j})_{L^2(\Omega)}\varphi_{n_j}(x_0)}{(p\Lambda_{n_j}+\widetilde\Lambda_{n_j})^k} 
\\\nonumber&= 
\sum_{k=1}^K \frac{(-1)^{k+1}}{\Gamma(1-\alpha k)}
\frac{1}{t^{\alpha k}}
\sum_{n=1}^\infty 
\sum_{j=1}^{d_n}
\frac{(u_{0,q},\varphi_{n,j})_{L^2(\Omega)}\varphi_{n_j}(x_0)}{(q\Lambda_{n_j}+\widetilde\Lambda_{n_j})^k}
+ O\left( \frac{1}{t^{\alpha(K+1)}}\right)
\end{align}
as $t \to \infty$.

By \eqref{eq:42410}, 
equating the coefficients of $t^{-\alpha k}$, 
we have
\begin{equation*}
\sum_{n=1}^\infty 
\sum_{j=1}^{d_n}
\frac{(u_{0,p},\varphi_{n,j})_{L^2(\Omega)}\varphi_{n_j}(x_0)}{(p\Lambda_{n_j}+\widetilde\Lambda_{n_j})^k}
= 
\sum_{n=1}^\infty 
\sum_{j=1}^{d_n}
\frac{(u_{0,q},\varphi_{n,j})_{L^2(\Omega)}\varphi_{n_j}(x_0)}{(q\Lambda_{n_j}+\widetilde\Lambda_{n_j})^k}
,
\quad k\in \mathbb{N}.
\end{equation*}
Assume that $\frac{q}{p}>1$. 
Then we have 
\begin{equation*}
\frac{P_1 u_{0,p}(x_0)}{(p\Lambda_1+\widetilde\Lambda_1)^k}
+ 
\sum_{n=2}^{\infty} 
\sum_{j=1}^{d_n}
\frac{(u_{0,p},\varphi_{n,j})_{L^2(\Omega)}\varphi_{n_j}(x_0)}{(p\Lambda_{n_j}+\widetilde\Lambda_{n_j})^k}
= 
\sum_{n=1}^\infty 
\sum_{j=1}^{d_n}
\frac{(u_{0,q},\varphi_{n,j})_{L^2(\Omega)}\varphi_{n_j}(x_0)}{(q\Lambda_{n_j}+\widetilde\Lambda_{n_j})^k}
.
\end{equation*}
Here $p\Lambda_1+\widetilde\Lambda_1$ is the first eigenvalue and this is simple.
Hence,
\begin{align*}
&P_1 u_{0,p}(x_0)
+ \sum_{n=2}^{\infty} 
\sum_{j=1}^{d_n}
\left( \frac{p\Lambda_{1}+\widetilde\Lambda_{1}}{p\Lambda_{n_j}+\widetilde\Lambda_{n_j}}\right)^k 
(u_{0,p},\varphi_{n,j})_{L^2(\Omega)}\varphi_{n_j}(x_0)
\\&
= \sum_{n=1}^\infty \left( \frac{p\Lambda_{1}+\widetilde\Lambda_{1}}{q\Lambda_{n_j}+\widetilde\Lambda_{n_j}}\right)^k 
(u_{0,q},\varphi_{n,j})_{L^2(\Omega)}\varphi_{n_j}(x_0).
\end{align*}
Since 
$\left| 
\frac{p\Lambda_{1}+\widetilde\Lambda_{1}}{p\Lambda_{n_j}+\widetilde\Lambda_{n_j}}
\right| < 1$ 
for $n\ge 2$ and
$\left| 
\frac{p\Lambda_{1}+\widetilde\Lambda_{1}}{q\Lambda_{n_j}+\widetilde\Lambda_{n_j}}
\right| < 1$ for 
$n\ge 1$
by the inequality \eqref{eq:42409}
, we see
that 
\begin{equation*}
\lim_{k\to \infty} \left( \frac{p\Lambda_{1}+\widetilde\Lambda_{1}}{p\Lambda_{n_j}+\widetilde\Lambda_{n_j}}\right)^k = 0 
\quad \text{\rm for $n\ge 2$}
\end{equation*}
and
\begin{equation*}
\lim_{k\to \infty} \left( \frac{p\Lambda_{1}+\widetilde\Lambda_{1}}{q\Lambda_{n_j}+\widetilde\Lambda_{n_j}}\right)^k = 0 
\quad \text{\rm for $n\ge 1$}.
\end{equation*}
Therefore we reach $P_1 u_{0,p}(x_0) = 0$. 
This is impossible by the assumption. 
Consequently $\frac{q}{p} \le 1$.  
We may also show $\frac{p}{q} \ge 1$ 
and $p=q$ follows.
Thus the proof of Theorem \ref{thm:32203} is complete.
\hfill\qed

%
%
\section*{Appendix}
We will prove here some lemmas 
which we have not proved yet.

%
%
\subsection{Proof of Lemmas \ref{lem:3101} and \ref{lem:3103}}
We give proofs of the two lemmas 
stated in the main result \S3.1.2 and \S3.1.2.
We will start with the proof of Lemma \ref{lem:3101}. 
\begin{proof}[Proof of Lemma \ref{lem:3101}]
Since $w_\xi=\xi w_1$ 
by the linearity of $w_\xi$ 
with respect to $\xi$, 
it is sufficient to show
\begin{equation*}
\M_Y(a\pp_y w_1)
=
\frac{1}{\M_Y\left(\frac{1}{a}\right)}.
\end{equation*}
Note that $w_1=y-\chi_1$, where 
$\chi_1$ is the weak solution of the problem:
\begin{equation*}
\left\{
\begin{aligned}
&
-\pp_y (a(y)\pp_y \chi_1)=-\pp_y a(y)\quad \text{in $\D^\prime(\mathbb{R})$},
\\&
\text{%
$\chi_1$ is $Y$-periodic and 
$\M_Y(\chi_1)=0$%
}.
\end{aligned}
\right.
\end{equation*}
With straightforward calculations, 
we may obtain $\chi_1$ of the following form:
\begin{equation*}
\chi_1=
y+C_1\int_0^y \frac{1}{a(z)}\,dz
+C_0
\end{equation*}
with constants $C_0$ and $C_1$. 
Since $\chi_1(0)=\chi_1(\ell_1)$ 
by the $Y$-periodicity of $\chi_1$, 
we have
\begin{equation*}
C_0
=
\ell_1
+
C_1\int_0^{\ell_1}
\frac{1}{a(z)}\,dz
+
C_0
.
\end{equation*}
Hence we see that 
\begin{equation*}
C_1
=
-\frac{\ell_1}{\int_0^{\ell_1}
\frac{1}{a(z)}\,dz}
=
-\frac{1}{\M_Y\left(\frac1{a}\right)}
.
\end{equation*}
Thus we obtain
\begin{equation*}
\chi_1=
y-\frac{1}{\M_Y\left(\frac1{a}\right)}
\int_0^y \frac{1}{a(z)}\,dz
+C_0,
\end{equation*}
where $C_0$ is a constant which satisfies 
$\M_Y(\chi_1)=0$.

By $w_1=y-\chi_1$, we have
\begin{equation*}
w_1=
\frac{1}{\M_Y\left(\frac1{a}\right)}
\int_0^y \frac{1}{a(z)}\,dz
-C_0,
\end{equation*}
and then 
\begin{equation*}
\pp_y w_1
=
\frac{1}{\M_Y\left(\frac1{a}\right)}
\cdot 
\frac{1}{a(y)}.
\end{equation*}
Using this equality, we may get
\begin{align*}
\M_Y(a\pp_y w_1)
&=
\M_Y\left(\frac{1}{\M_Y\left(\frac1{a}\right)}
\right)
\\&=
\frac{1}{\M_Y\left(\frac{1}{a}\right)}.
\end{align*}
Thus we conclude this lemma. 
\end{proof}
\begin{rmk}
In this case, 
we may obtain the homogenization result and this lemma directly without using the oscillating test function. 
\end{rmk}

Let us prove Lemma \ref{lem:3102}. 
\begin{proof}[Proof of Lemma \ref{lem:3102}]
Let $\chi_i$ be the unique solution of 
the problem \eqref{eq:2205} 
with $\xi=e_i$ $(i=1,\ldots,N)$:
\begin{equation}
\label{eq:a01}
\left\{
\begin{aligned}
&
-\dd_y (A(y)\nabla_y \chi_i)=-\dd_y(A(y)e_i)\quad \text{in $\D^\prime(\mathbb{R}^N)$},
\\&
\text{%
$\chi_i$ is $Y$-periodic and 
$\M_Y(\chi_i)=0$%
},
\end{aligned}
\right.
\end{equation}
and we set $w_i=y_i-\chi_i$ $(i=1,\ldots, N)$. 
Then $w_i$ is the unique solution of 
the problem \eqref{eq:2208} 
with $\xi=e_i$ $(i=1,\ldots,N)$. 

If $i=1$, then 
the coefficient and the non-homogeneous term in the above equation of the problem \eqref{eq:a01} 
depend on $y_1$ and are independent of $y_2,\ldots, y_N$. 
Hence we may find $\chi_1$ depending on only $y_1$, 
that is, the solution $\chi_1=\chi_1(y_1)$ of 
the following problem:
\begin{equation*}
\left\{
\begin{aligned}
&
-\frac{\pp}{\pp y_1} \left(a_{11}(y_1)\frac{\pp\chi_1}{\pp y_1}\right)=-\frac{\pp a_{11}}{\pp y_1}(y_1)\quad \text{in $\D^\prime(\mathbb{R})$},
\\&
\text{%
$\chi_1$ is $(0,\ell_1)$-periodic and 
$\M_{(0,\ell_1)}(\chi_1)=0$
}.
\end{aligned}
\right.
\end{equation*}
By the same argument in the proof of Lemma \ref{lem:3101}, we have
\begin{equation*}
\chi_1
=
y_1-
\frac{1}{\M_{(0,\ell_1)}\left(\frac1{a_{11}}\right)}
\int_0^{y_1}
\frac1{a_{11}(z)}
\,dz
+
C_{0,1}, 
\end{equation*}
where $C_{0,1}$ is a constant satisfying 
$\M_{(0,\ell_1)}(\chi_1)=0$. 

By $w_1=y_1-\chi_1$, we obtain
\begin{equation*}
w_1
=
\frac{1}{\M_{(0,\ell_1)}\left(\frac1{a_{11}}\right)}
\int_0^{y_1}
\frac1{a_{11}(z)}
\,dz
-
C_{0,1},
\end{equation*}
and then
\begin{equation*}
\nabla_y w_1
=
\begin{pmatrix}
\frac{1}{\M_{(0,\ell_1)}\left(\frac1{a_{11}}\right)}\cdot \frac1{a_{11}}
\\
0\\
\vdots\\
0
\end{pmatrix}
.
\end{equation*}

By \eqref{eq:3106} with $\xi=e_1$, we may see that
\begin{equation*}
\begin{pmatrix}
a_{11}^0\\
a_{21}^0\\
\vdots\\
a_{N1}^0
\end{pmatrix}
=
\begin{pmatrix}
\frac{1}{\M_{(0,\ell_1)}\left(\frac1{a_{11}}\right)}\\
\frac{1}{\M_{(0,\ell_1)}\left(\frac1{a_{11}}\right)}\M_{(0,\ell_1)}\left(\frac{a_{21}}{a_{11}}\right)\\
\vdots\\
\frac{1}{\M_{(0,\ell_1)}\left(\frac1{a_{11}}\right)}\M_{(0,\ell_1)}\left(\frac{a_{N1}}{a_{11}}\right)
\end{pmatrix}.
\end{equation*}
By the component of the above equality, we have
\begin{equation*}
a_{11}^0=a_{11}^\ast,\quad
a_{i1}^0=a_{i1}^\ast\quad i=2,\ldots,N.
\end{equation*}

Next we compute the solution $\chi_j$ of the problem \eqref{eq:a01} for $j=2,\ldots, N$. 
Let us calculate 
$\chi_j$ depending on only $y_1$, that is, 
the solution $\chi_j=\chi_j(y_1)$ of the 
following problem for $j=2,\ldots, N$:
\begin{equation*}
\left\{
\begin{aligned}
&
-\frac{\pp}{\pp y_1} \left(a_{11}(y_1)\frac{\pp\chi_j}{\pp y_1}\right)=-\frac{\pp a_{1j}}{\pp y_1}(y_1)\quad \text{in $\D^\prime(\mathbb{R})$},
\\&
\text{%
$\chi_j$ is $(0,\ell_1)$-periodic and 
$\M_{(0,\ell_1)}(\chi_j)=0$
}.
\end{aligned}
\right.
\end{equation*}
Using the same argument as used above, we have
\begin{equation*}
\chi_j
=
\int_0^{y_1}
\frac{a_{1j}(z)}{a_{11}(z)}
\,dz
-
\frac{\M_{(0,\ell_1)}\left(\frac{a_{1j}}{a_{11}}\right)}{\M_{(0,\ell_1)}\left(\frac{1}{a_{11}}\right)}
\int_0^{y_1}
\frac1{a_{11}(z)}
\,dz
+C_{0,j},
\end{equation*}
where $C_{0,j}$ are constants 
satisfying $\M_{(0,\ell_1)}(\chi_j)=0$ 
for $j=2,\ldots,N$.

By $w_j=y_j-\chi_j$, we obtain
\begin{equation*}
w_j
=
y_j
-
\int_0^{y_1}
\frac{a_{1j}(z)}{a_{11}(z)}
\,dz
+
\frac{\M_{(0,\ell_1)}\left(\frac{a_{1j}}{a_{11}}\right)}{\M_{(0,\ell_1)}\left(\frac{1}{a_{11}}\right)}
\int_0^{y_1}
\frac1{a_{11}(z)}
\,dz
-
C_{0,j},
\end{equation*}
and then
\begin{equation*}
\nabla_y w_j
=
\begin{blockarray}{(c)c}
-\frac{a_{1j}}{a_{11}}
+
\frac{1}{\M_{(0,\ell_1)}\left(\frac{1}{a_{11}}\right)}
\M_{(0,\ell_1)}\left(\frac{a_{1j}}{a_{11}}\right)
\cdot \frac{1}{a_{11}}
&\\
0&\\
\vdots&\\
1&\text{($j$-th component)}\\
\vdots&\\
0&
\end{blockarray}
.
\end{equation*}

By \eqref{eq:3106} with $\xi=e_j$, we may see that
\begin{equation*}
\begin{pmatrix}
a_{1j}^0\\
a_{2j}^0\\
\vdots\\
a_{Nj}^0
\end{pmatrix}
=
\begin{pmatrix}
\frac{1}{\M_{(0,\ell_1)}\left(\frac{1}{a_{11}}\right)}
\M_{(0,\ell_1)}\left(\frac{a_{1j}}{a_{11}}\right)
\\
\frac{1}{\M_{(0,\ell_1)}\left(\frac{1}{a_{11}}\right)}
\M_{(0,\ell_1)}\left(\frac{a_{1j}}{a_{11}}\right)
\M_{(0,\ell_1)}\left(\frac{a_{21}}{a_{11}}\right)
+
\M_{(0,\ell_1)}\left(a_{2j}-\frac{a_{1j}a_{21}}{a_{11}}\right)
\\
\vdots\\
\frac{1}{\M_{(0,\ell_1)}\left(\frac{1}{a_{11}}\right)}
\M_{(0,\ell_1)}\left(\frac{a_{1j}}{a_{11}}\right)
\M_{(0,\ell_1)}\left(\frac{a_{N1}}{a_{11}}\right)
+
\M_{(0,\ell_1)}\left(a_{Nj}-\frac{a_{1j}a_{N1}}{a_{11}}\right)
\end{pmatrix}
\end{equation*}
for $j=2,\ldots,N$. 

By the component of the above equality, we have
\begin{equation*}
a_{1j}^0=a_{1j}^\ast,\quad
a_{ij}^0=a_{ij}^\ast\quad i,j=2,\ldots,N.
\end{equation*}

Thus we complete the proof of Lemma \ref{lem:3102}.
\end{proof}

%
%
\section*{Acknowledgments}
Atsushi Kawamoto has been supported 
by Grant-in-Aid for Research Activity Start-up, 
Japan Society for the Promotion of Science (JSPS) KAKENHI Grant Number JP21K20333.

%

%
%

\end{document}